\documentclass{article}
	\usepackage[margin=25truemm]{geometry}
	\usepackage{amsthm}
	\newtheorem{theorem}{Theorem}
	\newtheorem{lemma}[theorem]{Lemma}
	\newtheorem{proposition}[theorem]{Proposition}
	\newtheorem{corollary}[theorem]{Corollary}
	\theoremstyle{definition}
	\newtheorem{definition}{Definition}

	\usepackage{amsmath,amssymb}
	\allowdisplaybreaks[3]
	\usepackage{mathtools}
	\mathtoolsset{showonlyrefs=true,showmanualtags=true}

	\renewcommand{\tilde}{\widetilde}
	\renewcommand{\hat}{\widehat}

	\DeclarePairedDelimiter{\HKparen}{(}{)}
	\DeclarePairedDelimiter{\HKbrack}{[}{]}
	\DeclarePairedDelimiter{\HKbrace}{\{}{\}}
	\DeclarePairedDelimiter{\HKangle}{\langle}{\rangle}
	\DeclarePairedDelimiter{\HKvert}{\lvert}{\rvert}
	\DeclarePairedDelimiter{\HKVert}{\lVert}{\rVert}

	\newcommand{\prn}{\HKparen}
	\newcommand{\sqb}{\HKbrack}
	\newcommand{\crb}{\HKbrace}
	\newcommand{\agb}{\HKangle}
	\newcommand{\abs}{\HKvert}
	\newcommand{\norm}{\HKVert}
	\DeclarePairedDelimiter{\lopen}{(}{]}
	
	\DeclarePairedDelimiterX{\setdef}[2]{\{}{\}}{#1\,\delimsize|\,#2}
	\makeatletter
	\newcommand{\DeclareDerivative}[2]{
		\def#1{\@ifnextchar[
			{\HK@derivative@a{#2}}
			{\HK@derivative@b{#2}}
		}}
	\def\HK@derivative@a#1[#2]#3#4{\frac{#1^{#2}#3}{#1 #4^{#2}}}
	\def\HK@derivative@b#1#2#3{\frac{#1 #2}{#1 #3}}
	\makeatother
	\DeclareDerivative{\dv}{\ddsymb}
	\DeclareDerivative{\pdv}{\partial}
	
	\newcommand{\bbR}{\mathbb{R}}
	\newcommand{\calA}{\mathcal{A}}
	\DeclareMathOperator{\supp}{supp}
	\DeclareMathOperator{\tr}{tr}

	\DeclareMathOperator{\HKdiv}{div}
	
	\renewcommand{\div}{\HKdiv}
	\newcommand{\ce}{e}
	\newcommand{\ddsymb}{d}
	\newcommand{\dd}{\mathop{}\!\ddsymb}	
	\newcommand{\iu}{i}
	\newcommand{\sub}[1]{_{\mathrm{#1}}}
	\newcommand{\laplacian}{\Delta}
	\newcommand{\fourier}{\mathcal{F}}
	\newcommand{\inv}{^{-1}}				
	
	\newcommand{\dummydot}{\>\cdot\>}
	
	\newcommand{\trp}[1]{{#1}^{\intercal}}
	\newcommand{\fractrans}{(\abs{\partial_t+y\cdot\nabla_x}^{6}+\abs{\nabla_x}^4)^{1/12}}
	\newcommand{\bddlin}{\mathbf{B}}

	\usepackage{hyperref}
\begin{document}

\title{Anisotropic maximal $L^p$-regularity estimates for a hypoelliptic operator}
\author{Kazuhiro Hirao\thanks{Department of Mathematics, 
	Graduate School of Science,
	Kyoto University, Kyoto, 606-8502, Japan\\
	E-mail address: hirao.kazuhiro.54s@st.kyoto-u.ac.jp}}
\date{}
\maketitle
\begin{abstract}
	We consider the maximal regularity 
	of a specific Vlasov-Fokker-Planck equation $\mathcal{A}u=f$ in
	the Euclidean space.
	The operator $\mathcal{A}=\Delta_{y}u-y\cdot \nabla_x{u}$ 
	is an example of the Ornstein-Uhlenbeck operators.
	We prove the existence of a solution 
	that satisfies the anisotropic maximal regularity estimates.
	To prove this we also show a similar estimates 
	and a weak (1, 1) estimate for $L=\partial_t-\mathcal{A}$, 
	which is of independent interest.
	Moreover, we show a maximal regularity estimate
	containing a fractional transport operator.
	These results rely on the pointwise estimates of 
	the fundamental solution of $L$.
\end{abstract}
\section{Introduction}
\subsection{Background and main results}
	Let us consider the Ornstein-Uhlenbeck operator in $\bbR^N$:
	\begin{equation}
		\calA=\div(A\nabla)+\agb{x,B\nabla}
		=\sum_{i,j=1}^N a_{ij}\partial_{x_ix_j}^2
		+\sum_{i,j=1}^N b_{ij}x_i\partial_{x_j},
	\end{equation}
	where $A=(a_{ij})$ is an $N\times N$ constant matrix
	that is symmetric positive semi-definite and
	$B=(b_{ij})$ is an $N\times N$ constant matrix.
	The operator $\calA$ is the infinitesimal generator of
	the Uhlenbeck-Ornstein semigroup, which is
	the Markov semigroup associated to a stochastic
	differential equation that describes a random motion
	of a particle in a fluid.

	H\"{o}rmander in \cite{Hormander1967} studied when
	$L\coloneqq \partial_t-\calA$ has hypoellipticity.
	Roughly speaking, a differential operator $P$ is called hypoelliptic
	if the smoothness of $P{u}$ implies the smoothness of $u$ itself.
	In \cite{Hormander1967},
	H\"{o}rmander showed that the operator $L$ is hypoelliptic
	if and only if
	\begin{equation}
		\int_0^t \exp(-s\trp{B}) A \exp(-s{B})\dd{s}>0
	\end{equation}
	for all $t>0$. Here, $\trp{B}$ denotes the transpose of $B$.
	Lanconelli and Polidoro~\cite{LanconelliPolidoro1994} proved that
	if $L$ is hypoelliptic,
	then we can set
	\begin{equation}
		A=\begin{pmatrix}
			A_0 & 0 \\ 0 & 0
		\end{pmatrix}
	\end{equation}
	with a $p_0\times p_0$ constant matrix $A_0$ $(p_0\le N)$ that is
	symmetric and positive definite
	and
	\begin{equation}
		B=\begin{pmatrix}
			*		& B_1		& 0		 & \dots	& 0		\\
			*  		& *			& B_2	 & \ddots	& \vdots\\
			\vdots	&  			& \ddots & \ddots 	& 0 	\\
			\vdots 	& 			& 		 & \ddots	& B_r	\\
			*		& \dots		& \dots	 & \dots	& *
		\end{pmatrix}
	\end{equation}
	where $B_j$ is a $p_{j-1}\times p_j$ block with
	rank $p_j$, $p_0\ge p_1\ge \dots \ge p_r\ge 1$ and
	$p_0+p_1+\dotsm+p_r=N$ for some basis of $\bbR^N$.
	We say that $L$ is degenerate if $p_0<N$ and
	is non-degenerate otherwise.

	In this paper,
	we study a special case of the Ornstein-Uhlenbeck operators,
	$N=2d$ and
	\begin{equation}
		A=
		\begin{pmatrix}
			0 & 0 \\ 0 & I
		\end{pmatrix},
		\qquad
		\text{and}
		\qquad
		B=
		\begin{pmatrix}
			0 & -I \\ 0 & 0
		\end{pmatrix}.
	\end{equation}
	Here $I$ is the $d\times d$ identity matrix.
	Then $L$ is hypoelliptic and degenerate.
	To focus on this case, we introduce the notations
	\begin{gather}
		x=(x_1,\dotsc,x_d),\qquad y=(y_1,\dotsc,y_d)=(x_{d+1},\dotsc,x_{2d}),\\
		\nabla_x=(\partial_{x_1},\dotsc,\partial_{x_d}),\qquad
		\nabla_y=(\partial_{y_1},\dotsc,\partial_{y_d}),\qquad
		\laplacian_y=\nabla_y\cdot \nabla_y.
	\end{gather}
	Then $\calA$ is written as
	\begin{equation}
		\calA{u}=\laplacian_{y}u-y\cdot \nabla_x{u}
		=\sum_{i=1}^d \partial_{y_i}^2u-\sum_{i=1}^dy_i\partial_{x_i}u.
	\end{equation}

	\begin{definition}
		(1)
		For given $f\in L\sub{loc}^1(\bbR^{2d}_{x,y})$,
		a function $u\in L\sub{loc}^1(\bbR^{2d}_{x,y})$ is called
		a weak solution of $\calA{u}=f$ in $\bbR^{2d}$ if
		\begin{equation}
			\agb{u,\laplacian_y\phi+y\cdot \nabla_x\phi}
			=\agb{f,\phi}
		\end{equation}
		holds for any test function
		$\phi\in C\sub{c}^{\infty}(\bbR^{2d}_{x,y})$.
		Here $\agb{f,g}=\int_{\bbR^{2d}}f(x,y)g(x,y)\dd{x}\dd{y}$.

		\noindent
		(2)
		For given $f\in L\sub{loc}^1(\bbR^{2d}_{x,y}\times\bbR_t)$,
		a function $u\in C(\bbR_t;L\sub{loc}^1(\bbR^{2d}_{x,y}))$
		is called a weak solution of $L{u}=f$ in $\bbR_{x,y}^{2d}\times\bbR_t$
		if
		\begin{equation}
			\agb{u(t),\phi(t)}
			+\int_s^t \agb{u,-\partial_{\tau}\phi-\laplacian_y{\phi}
				-y\cdot \nabla_x\phi}\dd{\tau}
			=\agb{u(s),\phi(s)}+\int_s^t\agb{f,\phi}\dd{\tau}
		\end{equation}
		holds for any $s,t\in \bbR$ with $s<t$ and
		any test function
		$\phi\in C\sub{c}^{\infty}(\bbR^{2d}_{x,y}\times\bbR_t)$.
	\end{definition}

	\begin{definition}
		For $s\in (0,2)$ and $f\in L^p(\bbR^d)$ with $1\le p<\infty$,
		$\abs{\nabla_x}^{s} f\in L^p(\bbR^d)$ means that
		\begin{equation}
			\abs{\nabla_x}^{s} f(x)
			=\frac{2^{s/2}\Gamma(\frac{d+s}{2})}{\pi^{s/2}\Gamma(-\frac{s}{2})}
			\lim _{r\to 0}
			\int_{\abs{y}>r}\frac{f(x+y)-f(x)}{\abs{y}^{d+s}}\dd{y}
			\qquad \text{in $L^p(\bbR^d)$}.
		\end{equation}
		This is equivalent to the definition by the Fourier transform
		if $1\le p\le 2$.
		See Kwa\'{s}incki~\cite{Kwasincki2017}, for example.
	\end{definition}

	\begin{definition}
		For $f\in L^2_{x,t}(\bbR^{d}\times\bbR)$ and 
		$y\in \bbR^d$, we write
		\begin{equation}
			(\abs{\partial_t+y\cdot\nabla_x}^{6}+\abs{\nabla_x}^4)^{1/12}f
			=\fourier\inv_{\xi,\tau}
			(\abs{\tau+y\cdot \xi}^6+\abs{\xi}^4)^{1/12} \fourier_{x,t}f.
		\end{equation}
	\end{definition}

	For $p,q\ge 1$,
	the $L_y^qL_{x,t}^p$-norm of a function $f$ is defined by
	\begin{equation}
		\norm{f}_{L_y^qL_{x,t}^p}
		=\norm*{\norm{f}_{L_{x,t}^p}}_{L_y^q}.
	\end{equation}
	The $L_y^qL_{x}^p$-norm is also defined in the similar way.

	As a main result of this paper,
	we shall prove the following maximal regularity estimates
	in the anisotropic $L^p$-spaces.
	\begin{theorem}\label{thm:main result A}
		Let $p,q\in (1,\infty)$.
		Assume that
		$f\in C\sub{c}^{\infty}(\bbR^{2d})$.
		Then there exists a weak solution of $\calA{u}=f$
		satisfying
		\begin{equation}\label{eq:main result A}
			\norm{\laplacian_y{u}}_{L_y^{q}L_{x}^{p}}
			+\norm{\abs{\nabla_x}^{2/3}u}_{L_y^{q}L_{x}^{p}}
			+\norm{\abs{\nabla_x}^{1/3}\nabla_y u}_{L_y^{q}L_{x}^{p}}
			\le C\norm{f}_{L_y^{q}L_{x}^{p}}
		\end{equation}
		for some constant $C=C(p,q,d)>0$.
	\end{theorem}

	The key step of the proof of Theorem \ref{thm:main result A} is
	the following maximal regularity estimates for the non-stationary problem,
	which has its own interest.

	\begin{theorem}\label{thm:main result L}
		Let $p,q\in (1,\infty)$.
		Assume that
		$f\in C\sub{c}^{\infty}(\bbR^{2d+1})$.
		Then there exists a weak solution of $L{u}=f$
		satisfying
		\begin{align}\label{eq:mani result L}
			\MoveEqLeft
			\norm{\laplacian_y{u}}_{L_y^{q}L_{x,t}^{p}}
			+\norm{\abs{\nabla_x}^{2/3}u}_{L_y^{q}L_{x,t}^{p}}
			+\norm{\abs{\nabla_x}^{1/3}\nabla_yu}_{L_y^{q}L_{x,t}^{p}}
			+\norm{(\abs{\partial_t+y\cdot\nabla_x}^{6}+\abs{\nabla_x}^4)^{1/12}
			\nabla_yu}_{L_y^{q}L_{x,t}^{p}}\\
			&\le C\norm{f}_{L_y^{q}L_{x,t}^{p}}
		\end{align}
		for some constant $C=C(p,q,d)>0$.
		Moreover, the corresponding weak (1,1) estimate also holds:
		\begin{align}\MoveEqLeft
			\norm{\laplacian_y{u}}_{L^{1,\infty}}
			+\norm{\abs{\nabla_x}^{2/3}u}_{L^{1,\infty}}
			+\norm{\abs{\nabla_x}^{1/3}\nabla_yu}_{L^{1,\infty}}
			+\norm{(\abs{\partial_t+y\cdot\nabla_x}^{6}+\abs{\nabla_x}^4)^{1/12}\nabla_yu}_{L^{1,\infty}}\\
			&\le C'\norm{f}_{L^1}.
		\end{align}
		Here, $C'$ is a positive constant depending only on the dimension.
	\end{theorem}

	When $p=q=2$,
	the estimate \eqref{eq:mani result L} 
	for $\laplacian_y{u}$ and $\abs{\nabla_x}^{2/3}u$
	is proved by Bouchut~\cite{Bouchut2002}
	as the a priori estimates.
	The proof relies on H\"{o}rmander's commutator
	\begin{equation}\label{eq:commutator}
		\nabla_x
		=[\nabla_y,\partial_t+y\cdot\nabla_x]
		=\nabla_y(\partial_t+y\cdot\nabla_x)
		-(\partial_t+y\cdot\nabla_x)\nabla_y
	\end{equation}
	and the energy method based on the integration by parts.
	Moreover, its method is also valid for the solution of $\calA{u}=f$ because
	\begin{equation}
		\nabla_x
		=[\nabla_y,y\cdot\nabla_x]
		=\nabla_y(y\cdot\nabla_x)-(y\cdot\nabla_x)\nabla_y
	\end{equation}
	holds.

	The $L^p$-estimate of $\laplacian_y{u}$
	for general hypoelliptic degenerate Ornstein-Uhlenbeck operators
	is proved by Bramanti, Cupini, Lanconelli and Priola~\cite{Bramanti2009}.
	However, since they take the case $\tr{B}\ne 0$ into consideration,
	they are forced to take $L^p$-norms
	only on the strip $S=\bbR^N\times [-1,1]$,
	rather than on $\bbR^N\times\bbR$.
	We extend the result of \cite{Bramanti2009} to
	the time-global case in exchange for specialization.

	It is a classical result
	proved in \cite{Rothschild1976HypoellipticDO}
	that if $\calA{u}=f$ and $f$ belongs to
	the Sobolev spaces $W^{\alpha,p}$ with $p\in (1,\infty)$ and $\alpha \ge 0$,
	then $u$ is in $W\sub{loc}^{\alpha+2/3,p}$.
	This result is very general because it is independent of
	the domain of the functions.
	When $\laplacian_y$ in $L$ is replaced
	by the fractional Laplacian
	$\laplacian_y^{\alpha/2}=-(-\laplacian_y)^{\alpha/2}$
	with $\alpha\in(0,2)$,
	the maximal regularity on both $x$ and $y$
	is proved by Chen and Zhang~\cite{Chen2018} and
	Huang, Menozzi and Priola~\cite{Huang2019162}.
	In \cite{Chen2018} the proof is based on
	the Fefferman-Stein type estimate
	that leads to the $L^{\infty}$-$\mathrm{BMO}$ boundedness.
	This approach is used also in \cite{Chen2019}
	to prove the maximal regularity for
	the Kolmogorov-type hypoelliptic operator
	with time-dependent coefficients.
	The proof in \cite{Huang2019162}
	is based on the H\"{o}rmander condition to derive the weak (1,1) estimate
	and similar to ours.
	However, our approach is based on the pointwise estimates for the 
	fundamental solution in a more explicit form than in \cite{Huang2019162}
	particularly in the estimate of $\abs{\nabla_x}^{2/3}\Gamma$;
	see \eqref{al:pointwise estimate} below.
	As for the optimal smoothing estimates in the H\"{o}lder spaces,
	the Shauder estimates are proved
	by Da Prato and Lunardi~\cite{DaPrato1995}
	in the non-degenerate case and
	by Lunardi~\cite{Lunardi1997}
	in the degenerate case.

	In any case, the previous studies introduced above
	\cite{Bramanti2009,Chen2019,Huang2019162,Rothschild1976HypoellipticDO}
	considered isotropic $L^p$-estimates.
	On the other hand,
	since the operator $\calA$ has an anisotropic structure
	with respect to $x$ and $y$,
	it is natural to study the anisotropic estimates for solutions
	to $\calA{u}=f$.
	For example, Dong and Yastrzhembskiy~\cite{Dong2022KFP}
	have considered anisotropic norms.
	We expect that the anisotropic $L^p$-estimates are useful 
	in the study of some non-linear problem 
	such as the triple deck equations arising from the boundary
	layer theory in the fluid mechanics, where the equations 
	contain $\calA$ as the principal linear term.
	
	We remark that 
	our result does not follow from \cite{Dong2022KFP} immediately
	because the estimates in \cite{Dong2022KFP} for the case we consider
	assumes $u$ and its derivative belong to $L^q_{y}L^p_{x,t}$,
	but the solution of $Lu=f, \ f\in L^q_{y}L^p_{x,t}$ 
	defined by the fundamental solution does not belong 
	to $L^q_{y}L^p_{x,t}$ in general.
	In addition,
	we show the estimate containing the fractional transport operator,
	i.e.,
	\begin{equation}
		\norm{\fractrans\nabla_yu}_{L^q_{y}L^p_{x,t}}
		\le C \norm{f}_{L^q_{y}L^p_{x,t}},
	\end{equation}
	which is not obtained in the literature even in the isotropic norms
	and is one of the novelty of this paper.
	We expect that this estimate is useful for the 
	study of the trace regularity of the solutions 
	and for the analysis of the Neumann boundary problem 
	of $Lu=f$ in the half space.

\subsection{Strategy of the proof}

	Our main goal is to establish the maximal regularity estimates
	in the anisotropic $L^p$-spaces for the stationary problem
	$\calA{u}=f$.
	However, it is not easy to treat the stationary problem directly
	except for the special case $L^2(\bbR^{2d})$,
	where the elementary energy method can be applied as shown
	in Bouchut~\cite{Bouchut2002}.
	We first consider the operator $L=\partial_t-\calA$.
	The fundamental solution $\Gamma$ of $L$ is
	constructed in \cite[p.\,148]{Hormander1967}.
	The key step is to derive the pointwise estimates 
	of $\Gamma$ and its derivatives as follows:
	if $l+m+n=0$ or $1$ and $\theta\in(0,1), \nu>0$, then
	\begin{align}\label{al:pointwise estimate}
		\begin{split}
		&\abs{\partial_{x_i}^l\partial_{y_j}^m \partial_t^n
			\laplacian_y\Gamma(x,y,t; x',y',t')}\\
		&\qquad  \le
		\frac{C}
		{(\abs{x-x'-(t-t')y'}^{1/3}+\abs{y-y'}+\abs{t-t'}^{1/2})^{4d+2+3l+m+2n}},\\
		&\abs{\partial_{x_i}^l\partial_{y_j}^m \partial_t^n
			\abs{\nabla_x}^{\theta}\Gamma(x,y,t; x',y',t')}\\
		&\qquad  \le
		\frac{C}{(\abs{x-x'-(t-t')y'}^{1/3}+\abs{y-y'}+\abs{t-t'}^{1/2})
		^{3d+3\theta+3l}
			(\abs{y-y'}+\abs{t-t'}^{1/2})^{d+m+2n}
		},\\
		&\abs{\partial_{x_i}^l\partial_{y_j}^m \partial_t^n
			\fractrans\Gamma(x,y,t; x',y',t')}\\
		&\qquad  \le
		\frac{C}{(\abs{x-x'-(t-t')y'}^{1/3}+\abs{y-y'}+\abs{t-t'}^{1/2})
		^{3d+3+3l+2n-\nu}\abs{y-y'}^{d-1+m-\nu}
		},
		\end{split}
	\end{align}
	which give us the H\"{o}rmander condition for the singular integral kernel
	defined by the fundamental solution, and thus,
	the theory of generalized Calder\'{o}n-Zygmund operator can be applied
	to obtain the estimate
	in the isotropic space $L^p(\bbR^{2d+1}), 1<p<\infty$;
	See Proposition \ref{prop:pointwise estimte gamma}, 
	\ref{prop:pointwise estimate gamma2,3} and 
	\ref{prop:pointwise estimate gamma4}
	in Section \ref{sec:Pointwise estimates for the fundamental solution}
	for details.
	The pointwise estimates \eqref{al:pointwise estimate} have
	its own interest. In particular, the estimate of
	$\abs{\nabla_x}^{2/3}\Gamma$ has a different form from the
	one of $\laplacian_y\Gamma$ due to the non-locality and
	the singularity in the Fourier side coming from the fractional operator
	$\abs{\nabla_x}^{2/3}$.
	This situation also differs from $\nabla_x\Gamma$, 
	because it has the estimate 
	\begin{align}\MoveEqLeft
		\abs{\partial_{x_i}^l\partial_{y_j}^m \partial_t^n
			\nabla_x\Gamma(x,y,t; x',y',t')}\\
		&\qquad  \le
		\frac{C}
		{(\abs{x-x'-(t-t')y'}^{1/3}+\abs{y-y'}+\abs{t-t'}^{1/2})^{4d+3+3l+m+2n}}
	\end{align}
	where $l+m+n\le 1$, 
	which can be shown in the same way as in the proof of
	Proposition \ref{prop:pointwise estimte gamma}.
	Clarifying this point as in
	\eqref{al:pointwise estimate} is one of the novelty of this paper,
	compared with the known works cited above, and 
	the explicit bounds \eqref{al:pointwise estimate} are
	useful in the study of the equations $Lu=f$ and $\calA{u}=f$.
	For example, in order to obtain the estimates in the anisotropic spaces,
	we need to apply the theory 
	of the vector-valued singular integral operators.
	The pointwise estimates \eqref{al:pointwise estimate} play
	a crucial role in this paper to derive the estimates 
	for the vector-valued kernel.

	We return to estimation for the solution of
	the stationary problem $\calA{u}=f$.
	If we assumed that the solution $u$ has compact support in advance,
	the simple argument in \cite{Bramanti2009} would be available.
	Indeed, let $\chi\colon\bbR\to\bbR$ be a smooth cut-off function
	with compact support,
	and define $\chi_R$ as $\chi_R(t)=\chi(t/R)$ for $R>0$.
	Let $u$ be the solution of $\calA{u}=f$ with compact support.
	Applying the estimate for the solution of $Lu=f$ to $u\chi_R$, we have
	\begin{equation}
		\norm{\laplacian_y{u}\chi_R}_{L^p}
		\le C\norm{L(u\chi_R)}_{L^p}
		\le C\prn{\norm{f\chi_R}_{L^p}+\norm{u\chi_R'}_{L^p}}.
	\end{equation}
	Dividing the both side by $R^{1/p}$ and taking $R\to\infty$
	yield the desired result.
	However,
	the above argument works only when $u$ is assumed to be
	compactly supported.
	Therefore, we need to apply another argument to show
	Theorem \ref{thm:main result A},
	where $u$ is not necessary compactly supported.
	Our approach is based on time-averaging.
	We construct a solution of $Lu_R=f\chi_R$ for $R>0$
	by the fundamental solution
	and then take the time-average of $u_R$ in the time-interval $[-R,R]$,
	written $U_R$.
	This $U_R$
	has anisotropic uniform estimates and
	converges the solution of $-\calA{u}=f$.

	The rest of the paper is organized as follows.
	Section \ref{sec:preliminaries} is devoted to preliminaries.
	We recall the Lie group and norm associated to $L$.
	We devote Section 
	\ref{sec:Pointwise estimates for the fundamental solution} to estimation 
	of the integral kernels.
	In Section \ref{sec:proof of Lp estimate},
	we prove Theorem \ref{thm:main result L} for the case of $p=q$
	by studying properties of the fundamental solution.
	Theorem \ref{thm:main result L} of the general case is
	proved in Section \ref{sec:proof of LyLxy estimate}.
	Finally, we show
	Theorem \ref{thm:main result A}
	in Section \ref{sec:proof for stationary solutions}.

	We adopt the usual convention ``$\lesssim$''.
	Namely, we write $A\lesssim B$ if there exists a positive constant $C$
	such that $A\le CB$
	and write $A\sim B$ if both $A\lesssim B$ and $B\lesssim A$ hold.

\section{Some known and preliminary results}
	\label{sec:preliminaries}
\subsection{Group and norm}

	Let us consider the operator $L$.
	It is left-invariant with respect to the associated translations:
	\begin{equation}
		(x',y',t')\circ (x,y,t)
		=(x+x'+ty',y+y',t+t').
	\end{equation}
	The operation is the group 
	introduced in \cite{LanconelliPolidoro1994}
	and
	\begin{equation}
		(x',y',t')\inv\circ(x,y,t)=(x-x'-(t-t')y',y-y',t-t').
	\end{equation}

	Furthermore, set $\delta(\lambda)z=(\lambda^3x,\lambda y,\lambda^2 t)$
	for $\lambda>0$.
	Then $L(u\circ\delta(\lambda))=\lambda^2(Lu)\circ \delta(\lambda)$.
	According to the scaling,
	we define the ``norm''
	\begin{equation}
		\norm{z}=\abs{x}^{1/3}+\abs{y}+\abs{t}^{1/2}
	\end{equation}
	for $z=(x,y,t)\in \bbR^{2d+1}$.
	By Young's inequality\footnote{
		In this paper, 
		we use the variables $x,y,t$ for the components of
		$z \in\bbR^d\times\bbR^d\times\bbR$ without any mention
		and use $z'=(x',y',t'), \zeta=(\xi,\eta,\tau)$ etc.
		in the similar way.},
	\begin{align}
		\norm{(z')\inv}
		& =\norm{(-x'+t'y',-y',-t')}
		 \le \abs{x'}^{1/3}+\frac{2}{3}\abs{t'}^{1/2}+\frac{1}{3}\abs{y'}
		+\abs{y'}+\abs{t'}^{1/2}\\
		& \le \frac{5}{3}\norm{z'}.
	\end{align}
	Similarly,
	\begin{equation}
		\norm{z'\circ z}
		=\norm{(x+x'+ty',y+y',t+t')}
		\le \frac{5}{3}\prn{\norm{z}+\norm{z'}}.
	\end{equation}
	We denote the constant $5/3$ by $c_0$.

	Define a quasi-symmetric quasi-distance $d$ to be
	\begin{equation}
		d(z,z')
		 =\norm{(z')\inv\circ z}
		 =\abs{x-x'-(t-t')y'}^{1/3}+\abs{y-y'}+\abs{t-t'}^{1/2}.
	\end{equation}
	The quasi-symmetric quasi-distance $d$ enjoys the properties
	\begin{gather}
		d(z',z)\le \frac{5}{3}d(z,z') ,\qquad
		d(z,z')\le \frac{5}{3}(d(z,z'')+d(z'',z')).
	\end{gather}

	We denote by $B(z',\delta)$ the ``ball'' induced by $d$
	with center $z'$ and radius $\delta$,
	that is $B(z',\delta)=\setdef{z\in \bbR^d}{d(z,z')<\delta}$.
	\begin{proposition}
		Each ball $B(z,\delta)$ is open to the Euclidean topology and
		the topology induced by $d$ coincides the Euclidean topology.
	\end{proposition}
	\begin{proof}
		It is clear that the ball $B(z',\delta)$ is open to the Euclidean
		topology and therefore $B(z',\delta)$ contains
		a Euclidean ball with center $z'$
		because the mapping $z\mapsto d(z,z')$ is continuous.

		Take $\epsilon>0$ and $z'\in \bbR^{2d+1}$.
		There exists $\delta>0$ such that 
		$\norm{(z')\inv\circ z}<\delta$ implies 
		\begin{equation}
			\norm{(z')\inv\circ z}
			+\norm{(z')\inv\circ z}^{2/3}\abs{y'}
			<\min\crb{\epsilon, 1}.
		\end{equation}
		Then $\abs{z-z'}<\epsilon$ holds for $z\in B(z',\delta)$.
		Indeed, if $z\in B(z',\delta)$, then
		\begin{equation}
			\min\crb{\epsilon, 1}
			> \norm{(z')\inv\circ z}
				+\norm{(z')\inv\circ z}^{2/3}\abs{y'}
			\ge \norm{z-z'}\ge \abs{z-z'}
		\end{equation}
		because 
		\begin{equation}
			\abs{x-x'}^{1/3}\le \abs{x-x'-(t-t')y'}^{1/3}+\abs{(t-t')y'}^{1/3}.
		\end{equation}
	\end{proof}

	\begin{proposition}\label{thm:triangle ineq below}
		If $0<M<c_0^{-2}$, then
		$\norm{z'}\le M\norm{z}$ implies
		\begin{equation}
			\frac{1-c_0^2M}{c_0}\norm{z}
			\le \norm{z\circ z'}.
		\end{equation}
	\end{proposition}
	\begin{proof}
		Use the triangle inequality and assumption to
		$z=(z\circ z')\circ (z')\inv$.
	\end{proof}
	\begin{corollary}\label{cor:triangle ineq for qd}
		There exist two constants $c_1,c_2>0$ such that
		if $c_1\min\crb{\norm{z'},\norm{{z'}\inv}}\le \norm{z}$,
		then $c_2\norm{z}\le \min\crb{\norm{z\circ z'},\norm{z'\circ z}}$.
	\end{corollary}

	\begin{proposition}
		For some constant $c>0$ we have
		$\abs{B(z,\delta)}=c\delta^{4d+2}$,
		which gives the doubling condition.
		Here, $\abs{B(z,\delta)}$ is the Lebesgue measure of 
		the ball $B(z,\delta)$.
	\end{proposition}
	\begin{proof}
		It is proved by straightforward computation.
	\end{proof}

	Note that the Lebesgue measure is invariant under inversion;
	namely, since $z\inv=(-x+ty,-y,-t)$,
	we have 
	\begin{equation}
		\int_{\bbR^{2d+1}}f(z\inv)\dd{z}=
		\int_{\bbR^{2d+1}}f(z)\dd{z}.
	\end{equation}

\subsection{Fundamental solution}

	The solution of $\partial_t{u}-\calA{u}=f$ is 
	formally given by
	\begin{equation}\label{eq:solution of L defined by fundamental solution}
		u(z)
		= (Tf)(z)
		\coloneqq \int_{-\infty}^t \int_{\bbR^{2d}}
		f(z')\Gamma(z,z')\dd{z'},
	\end{equation}
	where the fundamental solution $\Gamma$ is defined as
	\begin{align}
		\gamma(x,y,t) & =
		\begin{dcases}
			\prn*{\frac{3}{4\pi^2t^4}}^{d/2}
			\exp\prn*{-\frac{1}{t^3}\prn{3\abs{x}^2-3x\cdot ty+\abs{ty}^2}}
			& (t>0)\\
			\quad \ \,  0 & (t\le 0)
		\end{dcases}
		,\\
		\Gamma(z,z')& =\gamma((z')\inv\circ z).
	\end{align}
	The function $\gamma$ is represented by the Fourier transform. Indeed,
	\begin{equation}
		\gamma(x,y,t)=\fourier\inv\prn{\ce^{-F(\xi,\eta,t)} }
		=\frac{1}{(2\pi)^{2d}}
		\int_{\bbR^d\times\bbR^d} \ce^{-F(\xi,\eta,t)}
		\ce^{\iu(\xi\cdot x+\eta\cdot y)}\dd{\xi}\dd{\eta},
	\end{equation}
	where
	\begin{align}
		F(\xi,\eta,t)
		=\frac{1}{3}
		\prn[\big]{t^3\abs{\xi}^2+3t^2(\xi\cdot\eta)+3t\abs{\eta}^2}.
	\end{align}
	Observe that
	\begin{equation}
		F(\xi,\eta,t)\gtrsim t^3\abs{\xi}^2+t\abs{\eta}^2
	\end{equation}
	holds, which will be frequently used in this paper.

	To prove Theorem \ref{thm:main result L} let us study
	the operators
	\begin{align}
		T_1f&=\laplacian_y Tf,\\
		T_2f&=\abs{\nabla_x}^{2/3} Tf,\\
		T_3f&=\abs{\nabla_x}^{1/3}\nabla_yTf,\\
		T_4f
		&=(\abs{\partial_t+y\cdot\nabla_x}^{6}+\abs{\nabla_x}^4)^{1/12}
		\nabla_yTf.
	\end{align}
	Set
	\begin{gather}
		\gamma_1(z)=\laplacian_y\gamma(z), \qquad
		\gamma_2(z)=\abs{\nabla_x}^{2/3}\gamma(z),\qquad
		\gamma_3(z)=\abs{\nabla_x}^{1/3}\nabla_y\gamma(z),\\
		\gamma_4(z)
		=(\abs{\partial_t+y\cdot\nabla_x}^{6}+\abs{\nabla_x}^4)^{1/12}
		\nabla_y\gamma(z),\\
		\Gamma_i(z,z')=\gamma_i((z')\inv \circ z).
	\end{gather}

\section{Pointwise estimates for the fundamental solution}
	\label{sec:Pointwise estimates for the fundamental solution}
	We can easily check pointwise estimates for 
	$\gamma$ and $\gamma_1$ because
	we know their exact expressions.
	\begin{proposition}\label{prop:pointwise estimte gamma}
		The following estimates hold:
		\begin{equation}\label{eq:poitwise estimate gamma}
			\abs{\partial_{x_i}^l\partial_{y_j}^m \partial_t^n\gamma(z)}
			\lesssim \frac{1}{\norm{z}^{4d+3l+m+2n}}.
		\end{equation}
		Here, $l,m,n$ are non-negative integers.
		In particular,
		\begin{equation}
			\abs{\gamma(z)}\lesssim \frac{1}{\norm{z}^{4d}},\qquad
			\qquad \abs{\nabla_y\gamma(z)}\lesssim \frac{1}{\norm{z}^{4d+1}}.
		\end{equation}
	\end{proposition}
	\begin{proof}
		We give a proof only for the case $(l,m,n)=(0,0,0)$.
		Other estimates can be proved in the similar way.
		From the Taylor expansion of $\ce^{x}$,
		we have $x^k/k!\le \ce^{k}$ for $x\ge 0$ and thus
		\begin{align}
			\abs{x}^{4d/3}\abs{\gamma(z)}
			& \lesssim \prn*{\frac{\abs{x}^2}{t^3}}^{2d/3}
			\exp\prn*{-\frac{1}{t^3}(3\abs{x}^2-3x\cdot ty+\abs{ty}^2)}
			\lesssim 1,\\
			\abs{y}^{4d}\abs{\gamma(z)}
			& \lesssim \prn*{\frac{\abs{ty}^2}{t^3}}^{2d}
			\exp\prn*{-\frac{1}{t^3}(3\abs{x}^2-3x\cdot ty+\abs{ty}^2)}
			\lesssim 1,\\
			\abs{t}^{2d}\abs{\gamma(z)}
			& \lesssim
			\exp\prn*{-\frac{1}{t^3}(3\abs{x}^2-3x\cdot ty+\abs{ty}^2)}
			\lesssim 1.
		\end{align}
		Hence $\norm{z}^{4d}\abs{\gamma(z)}$ is bounded from above.
	\end{proof}
\subsection{Estimates for \texorpdfstring{$\gamma_2$}{gamma2} and \texorpdfstring{$\gamma_3$}{gamma3}}

	\begin{lemma}\label{lem:decay of s-derivative}
		Define the set $V_c$ as
		\begin{equation}
			V_c=\operatorname{span}\setdef*{
				\frac{\xi^{\alpha}\eta^{\beta}}
				{t^{(c-3\abs{\alpha}-\abs{\beta})/2}}
				\ce^{-F(\xi,\eta,t)}
			}{\text{$\alpha,\beta$: multiindices}}
		\end{equation}
		for an integer $c$.
		If $c$ is non-negative and $f\in V_c$, then
		\begin{equation}\label{eq:decay of s-derivative}
			\abs*{\int \xi^{\gamma} \abs{\xi}^{\theta}
			f(\xi,\eta,t)\ce^{\iu(x\cdot\xi+y\cdot \eta)}
				\dd{\xi}\dd{\eta}}
			\lesssim \frac{1}{\norm{z}^{3d+3\abs{\gamma}+3\theta}
				(\abs{y}+\abs{t}^{1/2})^{d+c}}
		\end{equation}
		for any multiindex $\gamma$ and $\theta \in(0,1)$.
	\end{lemma}
	\begin{proof}
		Note that 
		$\nabla_\xi f\in V_{c-3}$, $\nabla_\eta f\in V_{c-1}$ and 
		$\partial_t f\in V_{c-2}$ if $f\in V_c$.
		Moreover, we have the estimate
		\begin{equation}\label{eq:decay for Phi}
			\int_{\bbR^{2d}}
			\abs*{\xi^{\gamma} \abs{\xi}^{\theta} f(\xi,\eta,t)}\dd{\xi}\dd{\eta}
			\lesssim \frac{1}{t^{(4d+3\abs{\gamma}+3\theta+c)/2}}
		\end{equation}
		for $f\in V_c$.
		Take $f\in V_c$ and define $\Phi$ as 
		\begin{equation}
			\Phi(x,y,t)
			=\int \xi^{\gamma} \abs{\xi}^{\theta} 
				f(\xi,\eta,t)\ce^{\iu(x\cdot\xi+y\cdot \eta)}
			\dd{\xi}\dd{\eta}.
		\end{equation}
		Applying \eqref{eq:decay for Phi}
		for $\partial_{\eta_i}^{4d+2+c}f, t^{(4d+2+c)/2}f\in V_{-(4d+2)}$,
		we obtain 
		\begin{align}
			(\abs{y}+\abs{t}^{1/2})^{4d+2}\abs{\Phi(x,y,t)}
			\lesssim 1.
		\end{align}
		
		Let us find the the decay rate of $\Phi$ for $\abs{x}$.
		Take a $C^{\infty}$ function $\chi\colon \bbR^d\to \bbR$ 
		supported in the ball with center the origin and radius 2
		such that $\chi(\xi)=1$ for $\abs{\xi}\le 1$ and
		set $\phi_R(\xi)=\chi(\xi/R)$ for $R>0$.
		We decompose $\Phi$ using $1=\phi_R+(1-\phi_R)$.
		We have the estimate
		\begin{align}\MoveEqLeft
			\abs*{y_j^{d+c}\int_{\bbR^d}\int_{\bbR^d}
				\xi^{\gamma}\abs{\xi}^{\theta} f(\xi,\eta,t)\phi_R(\xi)
				\ce^{\iu(x\cdot\xi+y\cdot \eta)}\dd{\xi}\dd{\eta}}\\
			&\lesssim 
			\int_{\abs{\xi}\le 2R}\abs{\xi}^{\abs{\gamma}+\theta}
				\int_{\bbR^d}
				\abs*{\pdv[d+c]{f}{\eta_j}(\xi,\eta,t)}
				\dd{\eta}\dd{\xi}
			\lesssim 
				R^{d+\abs{\gamma}+\theta} 
				\norm*{\pdv[d+c]{f}{\eta_j}}_{L_{\xi,t}^{\infty} L_{\eta}^1}
		\end{align}
		because $\partial_{\eta_j}^{d+c}f\in V_{-d}$.
		Let $n$ be a non-negative integer. By the Leibniz rule,
		\begin{align}\MoveEqLeft 
			\pdv[n]{}{\xi_i}
			\prn*{\xi^{\gamma}\abs{\xi}^{\theta} \pdv[d+c]{f}{\eta_j}(\xi,\eta,t)
			(1-\phi_R(\xi))}\\
			& =\sum_{k_1+k_2+k_3=n}
			\frac{n!}{k_1! k_2! k_3!}
			\pdv[k_1]{\xi^{\gamma}\abs{\xi}^{\theta}}{\xi_i}
			\cdot
			\frac{\partial^{k_2+d+c}f}
			{\partial{\xi_i}^{k_2}\, \partial^{d+c}{\eta_j}}
			(\xi,\eta,t)
			\cdot \pdv[k_3]{}{\xi_i}(1-\phi_R(\xi)).
		\end{align}
		We observe
		\begin{align}
			\abs*{\pdv[k_1]{\xi^{\gamma}\abs{\xi}^{\theta}}{\xi_i}}
			&\lesssim \abs{\xi}^{\theta+\abs{\gamma}-k_1},
			\\
			\abs*{\int_{\bbR^{d}}\frac{\partial^{k_2+d+c}f}
			{\partial{\xi_i}^{k_2}\, \partial^m{\eta_j}}(\xi,\eta,t)\dd{\eta}}
			&\lesssim \abs{\xi}^{-k_2},
			\\
			\abs*{\pdv[k_3]{}{\xi_i}(1-\phi_R(\xi))}
			&\le\frac{1}{R^{k_3}} 
				\max_{\xi}\abs*{\pdv[k_3]{}{\xi_i}(1-\chi(\xi))}
			\lesssim \abs{\xi}^{-k_3};
		\end{align}
		hence if $n>d+s$, then
		\begin{align}\label{eq:estimate for Phi large xi}
			\MoveEqLeft
			\abs*{x_i^ny_j^{d+c}\int_{\bbR^d}\int_{\bbR^d}
				\xi^{\gamma}\abs{\xi}^{\theta} f(\xi,\eta,t)(1-\phi_R(\xi))
				\ce^{\iu(x\cdot\xi+y\cdot \eta)}\dd{\xi}\dd{\eta}}\\
			&\lesssim
			\int_{\abs{\xi}\ge R}
			\abs{\xi}^{\abs{\gamma}+\theta-n}\dd{\xi}
			\sim \frac{R^{d+\abs{\gamma}+\theta}}{R^n}.
		\end{align}
		Summing \eqref{eq:estimate for Phi large xi} over $i=1,\dotsc,d$
		yields
		\begin{equation}
			\abs*{y_j^{d+c}\int_{\bbR^d}\int_{\bbR^d}
				\abs{\xi}^{\theta} f(\xi,\eta,t)(1-\phi_R(\xi))
				\ce^{\iu(x\cdot\xi+y\cdot \eta)}\dd{\xi}\dd{\eta}}
			\lesssim \frac{R^{d+\abs{\gamma}+\theta}}{(\abs{x}R)^n}.
		\end{equation}
		Consequently, 
		\begin{equation}
			\abs{y_j^{d+c}\Phi(x,y,t)}\lesssim 
				R^{d+\abs{\gamma}+\theta}
				+\frac{R^{d+\abs{\gamma}+\theta}}{(\abs{x}R)^n}.
		\end{equation}
		Taking $R=1/\abs{x}$
		and summing this over $j=1,\dotsc,d$
		gives 
		\begin{equation}
			\abs{\Phi(x,y,t)}
			\lesssim \frac{1}{\abs{x}^{d+\abs{\gamma}+\theta}\abs{y}^{d+c}}.
		\end{equation}
		In the same manner we can see that 
		$\abs{\Phi(x,y,t)}
			\lesssim 1/\abs{x}^{d+\abs{\gamma}+\theta}\abs{t}^{(d+c)/2}$.
	\end{proof}
	\begin{proposition}\label{prop:pointwise estimate gamma2,3}
		If $l+m+n=0$ or $1$, then
		the following estimates hold:
		\begin{align}
			\abs{\partial_{x_i}^l\partial_{y_j}^m \partial_t^n\gamma_2(z)}
			\lesssim \frac{1}{\norm{z}^{3d+3l+2}
				(\abs{y}+\abs{t}^{1/2})^{d+m+2n}},\\
			\abs{\partial_{x_i}^l\partial_{y_j}^m \partial_t^n\gamma_3(z)}
			\lesssim \frac{1}{\norm{z}^{3d+3l+1}
				(\abs{y}+\abs{t}^{1/2})^{d+1+m+2n}}.
		\end{align}
	\end{proposition}
	\begin{proof}
		Apply Lemma \ref{lem:decay of s-derivative} to the derivatives of
		$\prn{\abs{\xi}^{\theta}\ce^{-F(\xi,\eta, t)}}^{\vee}$
		with $\theta=2/3$ for $\gamma_2$ and $\theta=1/3$ for $\gamma_3$.
	\end{proof}

\subsection{Estimates for \texorpdfstring{$\gamma_4$}{gamma4}}
	\label{subsec:estimates for gamma4}
	In this section,
	we show the pointwise estimate of $\gamma_4$.
	For a tuple of multiindices 
	$\alpha=(\alpha_1,\alpha_2,\alpha_3)$,
	we write
	\begin{equation}
		D^{\alpha}f
		=\partial_{x}^{\alpha_1}\partial_{y}^{\alpha_2}\partial_t^{\alpha_3}f,
		\qquad 
		\abs{\alpha}=3\abs{\alpha_1}+\abs{\alpha_2}+2\alpha_3.
	\end{equation}
	We are interested in the case of $\abs{\alpha}\ge 1$ 
	because our aim is the pointwise estimate of $\gamma_4$.
	We first investigate pointwise estimates 
	of the Fourier transform with respect to $(x,t)$ of $\gamma$.
	If a non-negative integer $n$ satisfies
	$d+\abs{\alpha}+12n-2>0$, then
	from Proposition \ref{prop:pointwise estimte gamma}
	\begin{align}
		\abs{(\abs{\tau+y\cdot \xi}^6+\abs{\xi}^4)^{n} \fourier 
			D^{\alpha}\gamma}
		&\sim \abs{\fourier (\prn{\partial_t+y\cdot \nabla_x}^{6n}
			+\prn{\laplacian_x}^{2n})
			D^{\alpha}\gamma}\\
		&\lesssim 
			\int\frac{1}{\norm{z}^{4d+\abs{\alpha}+12n}}\dd{x}\dd{t}\\
		&\lesssim 
			\frac{1}{\abs{y}^{d+\abs{\alpha}+12n-2}}.
	\end{align}
	Therefore, we have 	
	\begin{equation}\label{eq:estimate of gamma(x,t)-Fourier large mu}
		\abs{\fourier D^{\alpha}\gamma}
		\lesssim 
		\frac{1}{\abs{y}^{d+\abs{\alpha}+12\nu-2}
			(\abs{\tau+y\cdot \xi}^6+\abs{\xi}^4)^{\nu}}
	\end{equation}
	for any non-negative number $\nu$
	under the assumption of $d\ge 2$ and $\abs{\alpha}\ge 1$.
	However we need estimates of $\fourier D^{\alpha}\gamma$
	for $d=1$ and small $\nu$, which are not covered by 
	the computation above.
	We will see that \eqref{eq:estimate of gamma(x,t)-Fourier large mu}
	holds for strictly positive $\nu$ even if $d=1$.

	Since we already have \eqref{eq:estimate of gamma(x,t)-Fourier large mu}
	for large $\nu$, it suffices to show it for $\nu>0$ 
	arbitrarily small.
	Pick a cut-off function $\chi\in C\sub{c}^{\infty}(\bbR)$ 
	such that $\chi(t)=1$ for $\abs{t}\le 1$ and 
	$\chi(t)=0$ for $\abs{t}\ge 2$,
	and set $\phi_R(x,y,t)=\chi\prn[\big]{(\abs{x-ty}^4+\abs{t}^6)/R^{12}}$.
	We observe that
	\begin{align}
		\abs{(\partial_t+y\cdot\nabla_x)\phi_R(x,y,t)}
		&=\abs*{\chi'\prn*{\frac{\abs{x-ty}^4+\abs{t}^6}{R^{12}}}
			\frac{6t^5}{R^{12}}
		}\\
		&\lesssim \frac{1}{(\abs{x-ty}^4+\abs{t}^6)^{1/6}}
		\sim \frac{1}{\norm{(x-ty,0,t)}^{2}}.
	\end{align}
	Moreover, 
	\begin{equation}\label{eq:estimate phi_R}
	\begin{split}
		\abs{(\partial_t+y\cdot\nabla_x)^n\phi_R(x,y,t)}
		\lesssim \frac{1}{\norm{(x-ty,0,t)}^{2n}},\\
		\abs{\laplacian_x^n\phi_R(x,y,t)}
		\lesssim \frac{1}{\norm{(x-ty,0,t)}^{6n}}
	\end{split}
	\end{equation}
	hold for $n\ge 0$.
	Let $R>0, \abs{\alpha}\ge 1$ and $0<\nu<3d+2$.
	Then, 
	\begin{align}
		\abs{\fourier (\phi_R D^{\alpha}\gamma)}
		&\le \int_{\abs{x-ty}^4+\abs{t}^6\lesssim R^{12}}
			\frac{1}{\norm{z}^{4d+\abs{\alpha}}}\dd{x}\dd{t}\\
		&\le \int_{\abs{x}^4+\abs{t}^6\lesssim R^{12}}
			\frac{1}{\norm{(x+ty,t,y)}^{4d+\abs{\alpha}}}\dd{x}\dd{t}\\
		&\lesssim \int_{\abs{x}^4+\abs{t}^6\lesssim R^{12}}
			\frac{1}{\norm{z}^{4d+\abs{\alpha}}}\dd{x}\dd{t}\\
		&\le \frac{1}{\abs{y}^{d+\abs{\alpha}-2+\nu}}
			\int_{\abs{x}^4+\abs{t}^6\lesssim R^{12}} 
				\frac{1}{\norm{(x,0,t)}^{3d+2-\nu}}\dd{x}\dd{t}\\
		&\le \frac{R^{\nu}}{\abs{y}^{d+\abs{\alpha}-2+\nu}}.
		\label{al:estimate fourier D alpha gamma small area}
	\end{align}
	Here, we have used $\norm{z}\sim \norm{(x\pm ty,y,t)}$ in the 
	third line,
	which follows from Young's inequality.
	Fix $n>12/\nu$. We have
	\begin{align}\MoveEqLeft
		\abs[\big]{(\abs{\tau+y\cdot\xi}^6+\abs{\xi}^4)^n 
		\fourier \sqb{(1-\phi_R)D^{\alpha}\gamma}}\\
		&\sim \abs[\big]{\fourier\sqb{(\prn{\partial_t+y\cdot \nabla_x}^{6n}+
			\laplacian_x^{2n})(1-\phi_R)D^{\alpha}\gamma}}\\
		&\lesssim 
		\int_{\abs{x-ty}^4+\abs{t}^6\gtrsim R^{12}}
			\frac{1}{\norm{z}^{4d+\abs{\alpha}}
				\norm{(x-ty,0,t)}^{12n}
			}\dd{x}\dd{t}
			\qquad \text{by \eqref{eq:poitwise estimate gamma} and \eqref{eq:estimate phi_R}}\\
		&\lesssim 
			\int_{\abs{x}^4+\abs{t}^6\gtrsim R^{12}}
				\frac{1}{\norm{z}^{4d+\abs{\alpha}}\norm{(x,0,t)}^{12n}}
			\dd{x}\dd{t}
			\qquad  \text{by $\norm{(x+ty,y,t)}\sim \norm{z}$}\\
		&\lesssim 
			\frac{1}{\abs{y}^{d+\abs{\alpha}-2+\nu}}
			\int_{\abs{x}^4+\abs{t}^6\gtrsim R^{12}}
				\frac{1}{\norm{(x,0,t)}^{3d+2+12n-\nu}}\dd{x}\dd{t}\\
		&\lesssim 
			\frac{R^{\nu-12n}}{\abs{y}^{d+\abs{\alpha}-2+\nu}}.
		\label{al:estimate fourier D alpha gamma large area}
	\end{align}
	Letting $R=(\abs{\tau+y\cdot\xi}^6+\abs{\xi}^4)^{-1/12}$
	in \eqref{al:estimate fourier D alpha gamma small area}
	and \eqref{al:estimate fourier D alpha gamma large area}
	yields
	\begin{equation}
		\abs{\fourier D^{\alpha}\gamma}
		\lesssim \frac{1}{\abs{y}^{d+\abs{\alpha}-2+\nu}
			(\abs{\tau+y\cdot\xi}^6+\abs{\xi}^4)^{\nu/12}
		}.
	\end{equation}
	In conclusion, 
	we have \eqref{eq:estimate of gamma(x,t)-Fourier large mu}
	for any $\nu>0$ by interpolation.
	Especially,
	\begin{equation}\label{eq:estimate fourier gamma4}
		\abs{\fourier \nabla_y\gamma}
		\lesssim\frac{1}{\abs{y}^{d-1+\nu}
		(\abs{\tau+y\cdot\xi}^6+\abs{\xi}^4)^{\nu/12}
		}.
	\end{equation}
	By the similar computation, 
	we have for any $\nu>0$,
	\begin{align}\label{al:estimate gamma4 frequency}
		\abs{(\nabla_\xi-y\partial_\tau)^{\alpha} \fourier\nabla_y\gamma}
		&\sim 
		\abs[\big]{\fourier\sqb[\big]{
			(x-ty)^{\alpha}\nabla_y\gamma}}\\
		&\lesssim \frac{1}{\abs{y}^{d-1+\nu}
		(\abs{\tau+y\cdot\xi}^6+\abs{\xi}^4)^{(3\abs{\alpha}+\nu)/12}
		},\\
		\abs{\partial_\tau^n \fourier\nabla_y\gamma}
		&\sim \abs{\fourier \sqb{t^n \nabla_y\gamma}}\\
		&\lesssim \frac{1}{\abs{y}^{d-1+\nu}
			(\abs{\tau+y\cdot\xi}^6+\abs{\xi}^4)^{(2n+\nu)/12}
		}.
	\end{align}

	Let us show the pointwise estimate of $\gamma_4$.
	\begin{proposition}\label{prop:pointwise estimate gamma4}
		If $l+m+n=0$ or $1$, then
		\begin{equation}\label{eq:pointwise estimate gamma4}
			\abs{\partial_{x_i}^l\partial_{y_j}^m \partial_t^n\gamma_4(z)}
			\lesssim 
			\frac{1}{\abs{y}^{d-1+m+\nu}\norm{z}^{3d+3+3l+2n-\nu}}
		\end{equation}
	\end{proposition}
	\begin{proof}
		Let 
		$\psi_R(\xi,y,\tau)
			=\chi\prn[\big]{(\abs{\tau+y\cdot\xi}^6+\abs{\xi}^4)/R^{12}}$.
		Applying \eqref{eq:estimate fourier gamma4}
		for $\nu>0$ small enough, we obtain
		\begin{align}\MoveEqLeft
			\abs[\big]{\fourier\inv
				\sqb[\big]
				{\psi_R(\xi,y,\tau)\, (\abs{\tau+y\cdot\xi}^6+\abs{\xi}^4)^{1/12}
				\fourier_{x,t}\nabla_y\gamma}}\\
			&\lesssim\int_{\abs{\tau+y\cdot\xi}^6+\abs{\xi}^4\lesssim R^{12}}
				(\abs{\tau+y\cdot\xi}^6+\abs{\xi}^4)^{1/12}
				\abs{\fourier_{x,t}\nabla_y\gamma}\dd{\xi}\dd{\tau}\\
			&\lesssim \int_{\abs{\tau+y\cdot\xi}^6+\abs{\xi}^4\lesssim R^{12}}
				\frac{1}{\abs{y}^{d+1-2+\nu}}
				\prn{\abs{\tau+y\cdot \xi}^{1/2}+\abs{\xi}^{1/3}}^{1-\nu} \dd{\xi}\dd{\tau}\\
			&\lesssim 
				\frac{R^{3d+3-\nu}}{\abs{y}^{d+1-2+\nu}}.
				\label{al:estimate fourier gamma4 small}
		\end{align}
		Applying \eqref{eq:estimate fourier gamma4}
		for $\nu'>0$ large enough gives
		\begin{align}\MoveEqLeft
			\abs[\big]{\fourier\inv \sqb[\big]{(1-\psi_R)
				 (\abs{\tau+y\cdot\xi}^6+\abs{\xi}^4)^{1/12}
				\fourier_{x,t}\nabla_y\gamma}}\\
			&\lesssim \int_{\abs{\tau+y\cdot\xi}^6+\abs{\xi}^4\gtrsim R^{12}}
				(\abs{\tau+y\cdot\xi}^6+\abs{\xi}^4)^{1/12}
				\abs{\fourier_{x,t}\nabla_y\gamma}\dd{\xi}\dd{\tau}\\
			&\lesssim \int_{\abs{\tau+y\cdot\xi}^6+\abs{\xi}^4\lesssim R^{12}}
				\frac{1}{\abs{y}^{d+1-2+\nu}}
				\prn{\abs{\tau+y\cdot \xi}^{1/2}+\abs{\xi}^{1/3}}^{1-\nu} \dd{\xi}\dd{\tau}\\
			&\lesssim
			\frac{R^{3d+3-\nu'}}{\abs{y}^{d-1+\nu'}}.
		\end{align}
		Setting $R=\abs{y}^{-1}$ yields 
		\begin{equation}\label{eq:estimate gamma4 y}
			\abs{\gamma_4}
			\lesssim \frac{1}{\abs{y}^{4d+2}}.
		\end{equation}
		Fix $\nu>0$. If $n$ is large enough, then
		\begin{align}\MoveEqLeft
			\abs[\big]{(\abs{x-ty}^{4n}+\abs{t}^{6n})\fourier\inv 
				\sqb[\big]{(1-\psi_R)
				 (\abs{\tau+y\cdot\xi}^6+\abs{\xi}^4)^{1/12}
				\fourier_{x,t}\nabla_y\gamma}}\\
			&\lesssim \int_{\abs{\tau+y\cdot\xi}^6+\abs{\xi}^4\gtrsim R^{12}}
			\frac{1}{\abs{y}^{d-1+\nu}(\abs{\tau+y\cdot \xi}^{1/2}+\abs{\xi}^{1/3})^{\nu+12n-1}}\dd{\xi}\dd{\tau}\\
			&\lesssim \frac{R^{3d+2+1-12n-\nu}}{\abs{y}^{d-1+\nu}}.
		\end{align}
		Letting 
		$R=(\abs{\tau+y\cdot \xi}^{1/2}+\abs{\xi}^{1/3})^{-1}$ 
		in this estimate and \eqref{al:estimate fourier gamma4 small}
		yields
		\begin{equation}\label{eq:estimate gamma4 x-ty, t}
			\abs{\gamma_4(z)}\lesssim 
			\frac{1}{\abs{y}^{d-1+\nu}\norm{(x-ty,0,t)}^{3d+3-\nu}}.
		\end{equation}
		The estimates 
		\eqref{eq:estimate gamma4 y} and 
		\eqref{eq:estimate gamma4 x-ty, t}
		together with $\norm{z}\sim \norm{(x-ty,y,t)}$
		yield \eqref{eq:pointwise estimate gamma4} with $l=m=n=0$.
		Similarly,
		\begin{align}
			\abs{\nabla_x\gamma_4}
			&\lesssim
			\frac{1}{\abs{y}^{d-1+\nu}\norm{z}^{3d+6-\nu}}, \\
			\abs{(\partial_t+y\cdot\nabla_x)\gamma_4}
			&\lesssim
			\frac{1}{\abs{y}^{d-1+\nu}\norm{z}^{3d+5-\nu}}
		\end{align}
		and thus
		\begin{equation}
			\abs{\partial_t\gamma_4}
			\le \abs{(\partial_t+y\cdot\nabla_x)\gamma_4}
				+\abs{y\nabla_x\gamma_4}
			\lesssim
			\frac{1}{\abs{y}^{d-1+\nu}\norm{z}^{3d+5-\nu}}
		\end{equation}
	
		The estimate for $\nabla_y\gamma_4$	is slightly different.
		Indeed, 
		\begin{align}
			\nabla_y\gamma_4(z)
			&=
			\fourier\inv_{\xi,\tau}
				\sqb[\big]{
				\prn{\nabla_y(\abs{\tau+y\cdot\xi}^6+\abs{\xi}^4)^{1/12}}\,
				\fourier_{x,t}\nabla_y\gamma}
				+
			\fourier\inv_{\xi,\tau}
				\sqb[\big]{
				 (\abs{\tau+y\cdot\xi}^6+\abs{\xi}^4)^{1/12}
				\fourier_{x,t}\nabla_y^2\gamma}\\
			&\eqqcolon \tilde{\gamma}_{4,1}(z)+\tilde{\gamma}_{4,2}(z)
		\end{align}
		and 
		\begin{equation}
			\abs{\tilde{\gamma}_{4,1}(z)} \lesssim 
			\frac{1}{\abs{y}^{d-1+\nu}\norm{(x,y,t)}^{3d+4-\nu}}
		\end{equation}
		holds for $\nu>0$, but we have just
		\begin{equation}
			\abs{\tilde{\gamma}_{4,2}(z)}
			\lesssim 
			\frac{1}{\abs{y}^{d+\nu}\norm{(x,y,t)}^{3d+3-\nu}}
		\end{equation}
		for $\nu\ge 0$
		because we shall repeat the computation above
		using \eqref{eq:estimate of gamma(x,t)-Fourier large mu}
		with $\abs{\alpha}=2$.
		Therefore, 
		\begin{equation}
			\abs{\nabla_y\gamma_4(z)}
			\lesssim 
			\frac{1}{\abs{y}^{d+\nu}\norm{(x,y,t)}^{3d+3-\nu}}.
		\end{equation}
	\end{proof}

\section{Estimates for solutions to non-stationary problem}
	\label{sec:proof of Lp estimate}
\subsection{\texorpdfstring{$L^{2}$}{L2}-estimates}

	To obtain $L^2$-estimate, 
	we use the following theorem by Bouchut~\cite{Bouchut2002}.

	\begin{theorem}
		\label{thm:L2 bounded by Bouchut}
		If $Lu=f$ holds with
		\begin{equation}
			u,f\in L^2(\bbR^d_x\times\bbR^d_y\times\bbR_t),\qquad
			\nabla_y{u}\in L^2(\bbR^d_x\times\bbR^d_y\times\bbR_t),
		\end{equation}
		then $\laplacian_yu$ and $\abs{\nabla_x}^{3/2}u$ belong to
		$L^2(\bbR^d_x\times\bbR^d_y\times\bbR_t)$ and
		\begin{equation}
			\norm{\laplacian_yu}_2+\norm{\abs{\nabla_x}^{3/2}u}_2
			+\norm{\abs{\nabla_x}^{1/3}\nabla_yu}_2
			+\norm{(\abs{\partial_t+y\cdot\nabla_x}^{6}+\abs{\nabla_x}^4)^{1/12}
			\nabla_yu}_2
			\lesssim \norm{f}_2
		\end{equation}
		holds.
	\end{theorem}
	\begin{proof}
		See Theorem 1.5 in \cite{Bouchut2002} for the estimates 
		of $\laplacian_y{u}$ and $\abs{\nabla_x}^{2/3}u$.
		We obtain estimates of $\abs{\nabla_x}^{1/3}\nabla_yu$
		by interpolation, i.e., 
		\begin{align}
			\norm{\abs{\nabla_x}^{1/3}\nabla_yu}_2^2
			&\sim \norm{\abs{\xi}^{1/3}\eta\,\fourier_{x,y}{u}}_2^2
			= \sum_{i=1}^d
				\agb{\abs{\xi}^{1/3}\eta_i\,\fourier_{x,y}{u}, \abs{\xi}^{1/3}\eta_i\fourier_{x,y}{u}}\\
			&= \agb{\abs{\xi}^{2/3}\fourier_{x,y}{u}, \abs{\eta}^2\fourier_{x,y}{u}}
			\lesssim \norm{\abs{\nabla_x}^{2/3}u}_2\norm{\laplacian_y{u}}_2
			\lesssim \norm{f}_2^2.
		\end{align}
		Here, 
		$\agb{\dummydot, \dummydot}$ denotes 
		the inner product in $L^2(\bbR^{2d+1})$.

		Let $\hat{u}$ denote the Fourier transform of $u$ 
		with respect to $(x,t)$.
		It is sufficient to show the estimate 
		for $\abs{\partial_t+y\cdot\nabla_x}^{1/2}\nabla_yu$
		because of the triangle inequality in the frequency side.
		We have 
		\begin{align}
			\norm{\abs{\partial_t+y\cdot\nabla_x}^{1/2}\nabla_yu}_2^2
			&\sim \norm{\abs{\tau+y\cdot\xi}^{1/2}\nabla_yu}_2^2
			= \agb{\nabla_y\hat{u}, \abs{\tau+y\cdot\xi}\nabla_y\hat{u}}\\
			&\lesssim
				\sum_{i=1}^d 
				\abs{\agb{\hat{u},(\partial_{y_i}\abs{\tau+y\cdot\xi})\nabla_y\hat{u}}}
				+ \abs{\agb{\hat{u}, 
					\abs{\tau+y\cdot\xi}\laplacian_y\hat{u}}}.
				\label{al:L2 estimate integration by parts}
		\end{align}
		by integration by parts.
		Since 
		\begin{equation}
			\partial_{y_i}\abs{\tau+y\cdot\xi}=
			\begin{cases}
				\abs{\xi_i} &(\tau+y\cdot\xi>0)\\
				-\abs{\xi_i} &(\tau+y\cdot\xi<0)
			\end{cases},
		\end{equation}
		we have 
		\begin{equation}
			\abs{\agb{\hat{u},(\partial_{\xi_i}\abs{\tau+y\cdot\xi})\nabla_y\hat{u}}}
			\lesssim \norm{\abs{\xi_i}^{2/3}\hat{u}}_2
				\norm{\abs{\xi}^{1/3}\nabla_y\hat{u}}_2
			\lesssim \norm{\abs{\nabla_x}^{2/3}u}_2
				\norm{\abs{\nabla_x}^{1/3}\nabla_yu}_2
			\lesssim \norm{f}_2^2.
		\end{equation}

		We have to estimate the last term of 
		\eqref{al:L2 estimate integration by parts}.
		From the Cauchy-Schwartz inequality, 
		\begin{align}
			\abs{\agb{\hat{u}, 
					\abs{\tau+y\cdot\xi}\laplacian_y\hat{u}}}
			&=
			\abs*{\agb*{(\tau+y\cdot\xi)\hat{u},
				\frac{\tau+y\cdot\xi}{\abs{\tau+y\cdot\xi}}
				\laplacian_y{\hat{u}}}}
			\le \norm{(\tau+y\cdot\xi)\hat{u}}_2
				\norm{\laplacian_y\hat{u}}_2\\
			&\sim \norm{(\partial_t+y\cdot\nabla_x)u}_2
				\norm{\laplacian_y{u}}_2
			=\norm{\laplacian_y{u}+f}_2
				\norm{\laplacian_y{u}}_2
			\lesssim \norm{f}_2^2.
		\end{align}
	\end{proof}

	When $f$ is bounded and compactly supported, the conditions assumed in 
	Theorem \ref{thm:L2 bounded by Bouchut} 
	are verified as stated in the following proposition.
	\begin{proposition}\label{prop:Lp estimate for lower order term}
		If $f\in L^{\infty}(\bbR^d_x\times\bbR^d_y\times\bbR_t)$
		has compact support and
		$u$ is the solution defined by
		\eqref{eq:solution of L defined by fundamental solution},
		then $u,\nabla_y{u}\in L^p(\bbR^d_x\times\bbR^d_y\times\bbR_t)$ 
		for $p\in \lopen{1+1/2d,\infty}$.
		They especially belong to $L^2(\bbR^d_x\times\bbR^d_y\times\bbR_t)$ 
		in arbitrary dimensions.
	\end{proposition}
	\begin{proof}
		When $\alpha=4d,4d+1$,
		\begin{equation}
			\frac{4d+2}{p}<\alpha< 4d+2
		\end{equation}
		holds for $p\in \lopen{1+1/2d,\infty}$.

		We use $c_1,c_2$ in Corollary \ref{cor:triangle ineq for qd}.
		Let $L>0$ be large enough
		so that $\supp{f}\subset B(0,L)$.
		Let $R> c_1 L$.
		We estimate the norm of $u$ by dividing 
		the domain of integration into 
		the ball $B((0,0,0),R)$ and its complement.
		For the former domain, changing the variable, we have
		\begin{align}
			\norm*{\int f(z')\frac{1}{\norm{(z')\inv\circ z}^{\alpha}}
				\dd{z'}}_{L^p(B(0,R))}
			&=\norm*{\int f(z\circ (z'')\inv)
				\frac{1}{\norm{z''}^{\alpha}}\dd{z''}}_{L^p(B(0,R))} \\
			& \le \int_{B(0,c_0(R+c_0L))}\frac{1}{\norm{z''}^{\alpha}}\dd{z''}
			\norm{f}_{L^p}<\infty
		\end{align}
		because $\alpha<4d+2$.
		Here we used the fact 
		if $z\circ (z'')\inv\in \supp{f}\subset B(0,L)$, then
		\begin{equation}
			\norm{z''}
			\le c_0(c_0\norm{z\circ (z'')\inv}+\norm{z})
			\le c_0(c_0L+R).
		\end{equation}
		Next, we consider the latter domain.
		If $\norm{z'}\le L$ and $\norm{z}>R$,
		then $c_1\norm{z'}\le\norm{z}$ and thus
		$c_2\norm{z}\le \norm{(z')\inv\circ z}$.
		Therefore,
		\begin{align}
			\norm*{\int f(z')\frac{1}{\norm{(z')\inv\circ z}^{\alpha}}
				\dd{z'}}_{L^p(B(0,R)^c)}
			&\le \int \abs{f(z')}
				\norm*{\frac{1}{\norm{(z')\inv\circ z}^{\alpha}}}
					_{L^p(B(0,R)^c)}
				\dd{z'}\\
			& \lesssim
			\norm*{\frac{1}{\norm{z}^{\alpha}}}_{L^p(B(0,R))^c}\norm{f}_{L^1}
			<\infty
		\end{align}
		holds from $\alpha>(4d+2)/p$.
	\end{proof}
	Consequently, we can apply Theorem \ref{thm:L2 bounded by Bouchut} 
	and obtain the $L^2$-boundedness for $T_1$ and $T_2$.

\subsection{\texorpdfstring{$L^{p}$}{Lp}-estimates}
	\label{subsec:proof of general Lp estimate}

	The quasi-distance $d$ we are considering satisfies
	the assumptions of Chapter 1 of \cite{Stein1993HarmonicAnalysis},
	and therefore, the following theorem holds 
	from Theorem 1 in Chapter 1 of \cite{Stein1993HarmonicAnalysis}.
	
	\begin{theorem}
		\label{thm:SI Lp}
		Let $p_0\in (1,\infty)$.
		Assume that a bounded linear operator
		$T\colon L^{p_0}(\bbR^{2d+1})\to L^{p_0}(\bbR^{2d+1})$
		has an integral kernel
		$K\colon\bbR^{2d+1}\times\bbR^{2d+1}\setminus \Delta\to \bbR$ 
		such that if
		$f\in L^{p_0}(\bbR^{2d+1})$ has compact support then
		\begin{equation}
			Tf(z)=\int_{\bbR^{2d+1}}K(z,z')f(z')\dd{\mu(z')},
			\qquad z\notin\supp{f}.
		\end{equation}
		If there is a constant $c>1$ 
		such that
		\begin{equation}\label{eq:Hormander condition for ball}
			\int_{\norm{(z')\inv\circ z}\ge c\norm{(z'')\inv\circ z'}}
			\abs{K(z,z')-K(z,z'')}\dd{\mu(z)}
		\end{equation}
		is bounded by some constant independent of $z'$ and $z''$,
		then $T$ is weak $(1,1)$ and strong $(p,p)$ when $1<p<p_0$.

		Furthermore, If in addition,
		\begin{equation}
			\int_{\norm{(z')\inv\circ z}\ge c\norm{(z'')\inv\circ z'}}
			\abs{K(z',z)-K(z'',z)}\dd{\mu(z)}
		\end{equation}
		is bounded by some constant independent of $z'$ and $z''$,
		then $T$ is strong $(p,p)$ for all $p\in (1,\infty)$.
	\end{theorem}

	Because of the pointwise estimates of the derivatives of $\gamma_1$,
	it is easy to confirm the H\"{o}rmander condition
	\eqref{eq:Hormander condition for ball} for $T_1f=\laplacian_y Tf$
	as follows and so we can conclude the $L^p$-boundedness of $T_1$.

	\begin{proposition}\label{prop:difference estimate for Gamma1}
		We have
		\begin{align}
			\abs{\Gamma_1(z,z_1)-\Gamma_1(z,z_2)}
			& \lesssim
			\frac{\norm{z_2\inv\circ z_1}}{\norm{z_1\inv\circ z}^{4d+3}}, \\
			\abs{\Gamma_1(z_1,z)-\Gamma_1(z_2,z)}
			& \lesssim
			\frac{\norm{z_2\inv\circ z_1}}{\norm{z_1\inv\circ z}^{4d+3}}
		\end{align}
		if $c_1 \norm{z_2\inv\circ z_1}\le \norm{z_1\inv\circ z}$.
	\end{proposition}
	\begin{proof}
		Change the variables as
		$\zeta=z_1\inv\circ z,\zeta'=z_2\inv\circ z_1$.
		We should prove that
		\begin{align}\label{al:difference estimate for gamma1}
			\abs{\gamma_1(\zeta)
				-\gamma_1(\zeta'\circ \zeta)}
			\lesssim
			\frac{\norm{{\zeta'}}}{\norm{\zeta}^{4d+3}}
		\end{align}
		when $c_1\norm{\zeta'}\le \norm{\zeta}$.
		Setting $\zeta'=(\xi',\eta',\tau')$
		and applying the mean value theorem,
		we can find $\theta_1,\theta_2\in (0,1)$ such that
		\begin{align}\label{al:stretched gamma1}
			\MoveEqLeft
			\gamma_1(\zeta'\circ \zeta)
			-\gamma_1(\zeta)\\
			& =\gamma_1((\xi',\eta',\tau')\circ(\xi,\eta,\tau))
			-\gamma_1((0,0,\tau')\circ(\xi,\eta,\tau))\\
			&\qquad +\gamma_1((0,0,\tau')\circ(\xi,\eta,\tau))
			-\gamma_1(\xi,\eta,\tau)\\
			& =(\xi'\cdot\nabla_x+\eta'\cdot\nabla_y)
			\gamma_1((\theta_1\xi',\theta_1\eta',\tau')\circ \zeta)
			+\tau'(\eta \cdot\nabla_x+\partial_t)
			\gamma_1((0,0,\theta_2\tau')\circ\zeta).
		\end{align}
		Since
		\begin{equation}
			\norm{(\theta_1\xi',\theta_1\eta',\tau')}
			\le \norm{\zeta'}\le c_1\inv \norm{\zeta},\qquad
			\norm{(0,0,\theta_2\tau')}
			\le \norm{\zeta'}
			\le c_1\inv \norm{\zeta}
		\end{equation}
		we get
		\begin{align}
			\norm{(\theta_1\xi',\theta_1\eta',\tau')\circ \zeta}
			& \ge c_2\norm{\zeta}, \\
			\norm{(0,0,\theta_2\tau')\circ\zeta}
			& \ge c_2\norm{\zeta}
		\end{align}
		by the definition of $c_1, c_2$.
		Therefore,
		\begin{align}
			\abs{\xi'\cdot \nabla_x\gamma_1
				((\theta_1\xi',\theta_2\eta',\tau')\circ \zeta)}
			& \lesssim \frac{\abs{\xi'}}{\norm{
			((\theta_1\xi',\theta_1\eta',\tau')\circ \zeta)}^{4d+5}}\\
			& \lesssim \frac{\norm{\zeta'}^3}{\norm{\zeta}^{4d+5}}
			\lesssim
			\frac{\norm{\zeta'}}{\norm{\zeta}^{4d+3}}
		\end{align}
		and the other terms of the right-hand side of
		\eqref{al:stretched gamma1} can be estimated similarly.
	\end{proof}

	It is more delicate to check the H\"{o}rmander condition
	for $T_if\ (i=2,3,4)$ than $T_1$
	because $\gamma_i\ (i=2,3,4)$ do not enjoy
	the same derivative estimates that $\gamma_1$ does.

	\begin{proposition}\label{prop: Hormander condition 2, 3}
		The inequality \eqref{eq:Hormander condition for ball}
		holds for $K=\Gamma_2, \Gamma_3$ and its adjoint as $c=c_1$.
	\end{proposition}
	\begin{proof}
		It is enough to consider $\gamma_3$
		because $\gamma_2$ enjoys a better estimate than $\gamma_3$.
		We decompose the domain of integration into two parts:
		$\abs{x}^{1/3}\ge \norm{z}/2$ and 
		$\abs{y}+\abs{t}^{1/2}\ge \norm{z}/2$.
		For the latter domain,
		there is nothing difficult
		because we can use the same estimate as
		\eqref{al:difference estimate for gamma1}
		in virtue of Proposition \ref{prop:pointwise estimate gamma2,3}.
		
		Let us consider the domain
		where $\abs{x}^{1/3}\ge \norm{z}/2$.
		By the triangle inequality, 
		\begin{align}
			\abs{\gamma_3(z'\circ z)-\gamma_3(z)}
			& \le \abs{\gamma_3((x',y',t')\circ z)
			-\gamma_3((0,y',t')\circ z)}\\
			& \quad +\abs{\gamma_3((0,y',t')\circ z)
			-\gamma_3((0,0,t')\circ z)} \\
			& \quad +\abs{\gamma_3((0,0,t')\circ z)-\gamma_3(z)}.
		\end{align}
		We remark that 
		$\abs{x+\theta ty'}\gtrsim \abs{x}$ holds,
		because $\abs{\theta ty'}\le \norm{z}^2\norm{z'}\le \norm{z}^3/c_1
		\le \abs{x}(2^3/c_1)$.
		Thus,
		\begin{align}\MoveEqLeft
			\abs{\gamma_3((x',y',t')\circ z)
				-\gamma_3((0,y',t')\circ z)}\\
			& =\int_0^1 \abs{x'}\abs{\nabla_x\gamma_3(\theta x',y',t')\circ z}
			\dd{\theta}\\
			& \lesssim
			\frac{\abs{x'}}{\norm{(x,y+y',t+t')}^{3d+1+3}
				(\abs{y+y'}+\abs{t+t'}^{1/2})^{d+1}}
		\end{align}
		and hence, 
		\begin{align}
			\MoveEqLeft
			\int_{\abs{x}^{1/3}\ge (c_1/2)\norm{z'}} 
				\abs{\gamma_3((x',y',t')\circ z)
				-\gamma_3((0,y',t')\circ z)}\dd{z} \\
			& \lesssim \int_{\abs{x}^{1/3}\ge (c_1/2)\norm{z'}} 
				\frac{\abs{x'}}{\norm{(x,y+y',t+t')}^{3d+4}
				(\abs{y+y'}+\abs{t+t'}^{1/2})^{d+1}}\dd{z} \\
			& =\int_{\abs{x}^{1/3}\ge (c_1/2)\norm{z'}}
				\frac{\abs{x'}}{\abs{x}^{d+1}}\dd{x}
				\int_{\bbR^{d+1}}
				\frac{1}{\prn{1+\abs{\eta}+\abs{\tau}^{1/2}}^{3d+4}
					(\abs{\eta}+\abs{\tau}^{1/2})^{d+1}}
				\dd{\eta}\dd{\tau}
				\qquad \prn*{
					\begin{aligned}
						y+y'&=\abs{x}^{1/3}\eta,\\ t+t'&=\abs{x}^{2/3}\tau
					\end{aligned}
				}
			\\
			& \lesssim 1.
		\end{align}
		Next, we consider the difference with respect to $y$.
		Since the bound for 
		$(t\nabla_x+\nabla_y)\gamma_3$ is controlled 
		by the worse decay of $\nabla_y\gamma_3$, we have
		\begin{align}\label{al: gamma3 diff y}
			\abs{\gamma_3((0,y',t')\circ z)-\gamma_3((0,0,t')\circ z)}
			&\lesssim 
			\int_0^1 \abs{y'}\abs{((t+t')\nabla_x+\nabla_y)\gamma_3
				((\theta x',y',t')\circ z)
			}\dd{\theta}\\
			&\lesssim 
				\int_0^1\frac{\abs{y'}}{\norm{(0,\theta y',t')\circ z}^{3d+1}
				(\abs{y+\theta y'}+\abs{t+t'}^{1/2})^{d+2}}\dd{\theta}.
		\end{align}
		Again, we decompose the domain of integration
		into two parts: $\abs{y}\ge 2\abs{y'}$ and $2\abs{y'}\ge \abs{y}$.
		Fix $\sigma\in (0,1)$.
		We integrate 
		\begin{equation}
			(\abs{\gamma_3((0,y',t')\circ z)}^{1-\sigma}
			+\abs{\gamma_3((0,0,t')\circ z)}^{1-\sigma})
			\abs{\gamma_3((0,y',t')\circ z)-\gamma_3((0,0,t')\circ z)}^{\sigma}
			.
		\end{equation}
		In the domain of $\abs{y}\ge 2\abs{y'}$,
		replacing $\abs{y+\theta y'}$ and $\abs{y+y'}$ with $\abs{y}$
		yields
		\begin{align}\MoveEqLeft
			\prn*{\frac{1}{\norm{(x,y,t+t')}^{3d+1}(\abs{y}+\abs{t+t'}^{1/2})^{d+1}}}^{1-\sigma}
			\prn*{\frac{\abs{y'}}{\norm{(x,y,t+t')}^{3d+1}
				(\abs{y}+\abs{t+t'}^{1/2})^{d+2}}}^{\sigma}\\
			&\lesssim \frac{\abs{y'}^{\sigma}}{
				\norm{(x,y,t+t')}^{3d+1}
				(\abs{y}+\abs{t+t'}^{1/2})^{d+1+\sigma}
			}.
		\end{align}
		Hence,
		\begin{align}\MoveEqLeft
			\int_{\abs{x}^{1/3}\ge \norm{z}/2\gtrsim \norm{z'}}
				\frac{\abs{y'}^{\sigma}}{
				\norm{(x,y,t+t')}^{3d+1}
				(\abs{y}+\abs{t+t'}^{1/2})^{d+1+\sigma}
			}\dd{z}\\
			&\lesssim 
			\int_{\abs{x}^{1/3}\gtrsim\norm{z'}}
				\frac{\abs{y'}^{\sigma}}{\abs{x}^{d+\sigma/3}}\dd{x}
				\int \frac{1}{(1+\abs{y}+\abs{t}^{1/2})^{3d+1}(\abs{y}+\abs{t}^{1/2})^{d+1+\sigma}}\dd{y}\dd{t}\\
			&\lesssim 1.
		\end{align}
		In the domain of $2\abs{y'}\ge \abs{y}$,
		we estimate the integral without the derivatives.
		Fix $0<\nu<1$.
		Using $3\abs{y'}\ge \abs{y+y'}$, which is derived from 
		$2\abs{y'}\ge \abs{y}$, we obtain
		\begin{align}\MoveEqLeft
			\int_{2\abs{y'}\ge \abs{y},
				\abs{x}^{1/3}\ge \norm{z}/2\gtrsim \norm{z'}}
				\abs{\gamma_3((0,y',t')\circ z)}\dd{z}\\
			&\lesssim 
			\int 
			\frac{1}{\norm{(x,y+y',t+t')}^{3d+1}(\abs{y+y'}+\abs{t+t'}^{1/2})^{d+1}}\dd{z}\\
			&\lesssim
			\int
			\frac{1}{\norm{(x,0,0)}^{3d+1-\nu}(\abs{y+y'}+\abs{t+t'}^{1/2})^{d+1+\nu}}\dd{z}\\
			&\lesssim 
			\int \frac{1}{(\abs{x}^{1/3})^{3d+1-\nu}\abs{y+y'}^{d-1+\nu}}
			\dd{x}\dd{y}\\
			&\lesssim\int \frac{\abs{y'}^{1-\nu}}{\abs{x}^{d+(1-\nu)/3}}\dd{x}
			\lesssim \frac{\abs{y'}^{1-\nu}}{\norm{z'}^{1-\nu}}
			\lesssim 1.
		\end{align}
		Similar arguments apply to the estimate of 
		$\gamma_3((0,0,t')\circ z)$.

		Finally, we estimate $\abs{\gamma_3((0,0,t')\circ t)-\gamma_3(z)}$.
		We decompose the domain of integration into
		two parts: $\abs{t}\ge2\abs{t'}$ and $2\abs{t'}\ge \abs{t}$.
		In estimation in the domain of $\abs{t}\ge 2\abs{t'}$, 
		we can replace 
		$\abs{t+t'}, \abs{t+\theta t'}$ by $\abs{t}$
		and thus we have
		\begin{align}
			\abs{\gamma_3((0,0,t')\circ t)-\gamma_3(z)}
			&\lesssim \int_0^1\abs{t'}\abs{\partial_t\gamma_3(x,y,t+\theta t')}\\
			&\lesssim \frac{\abs{t'}}{\norm{z}^{3d+1}(\abs{y}+\abs{t}^{1/2})^{d+3}}
		\end{align}
		and
		\begin{equation}
			\abs{\gamma_3((0,0,t')\circ t)}+\abs{\gamma_3(z)}
			\lesssim \frac{\abs{t'}}{\norm{z}^{3d+1}(\abs{y}+\abs{t}^{1/2})^{d+1}}.
		\end{equation}
		Hence,
		\begin{align}
			\abs{\gamma_3((0,0,t')\circ t)-\gamma_3(z)}
			&\lesssim 
			\prn*{\frac{\abs{t'}}{\norm{z}^{3d+1}(\abs{y}+\abs{t}^{1/2})^{d+1}}}^{1-\sigma}
			\prn*{\frac{\abs{t'}}{\norm{z}^{3d+1}(\abs{y}+\abs{t}^{1/2})^{d+3}}}^{\sigma}\\
			&=\frac{\abs{t'}^{\sigma}}{\norm{z}^{3d+1}(\abs{y}+\abs{t}^{1/2})^{d+1+2\sigma}}.
		\end{align}
		Fix $\sigma \in(0,1/2)$.
		Then,
		\begin{align}
			\int_{\abs{t}\ge 2\abs{t'},
				\abs{x}^{1/3}\ge \norm{z}/2\gtrsim\norm{z'}}
				\abs{\gamma_3((0,0,t')\circ t)-\gamma_3(z)}\dd{z}
			&\lesssim 
				\int \frac{\abs{t'}^{\sigma}}{\norm{z}^{3d+1}(\abs{y}+\abs{t}^{1/2})^{d+1+2\sigma}}\dd{z}\\
			&\lesssim \int \frac{\abs{t'}^{\sigma}}{(\abs{x}^{1/3}+\abs{t}^{1/2})^{3d+1}\abs{t}^{1/2+\sigma}}
				\dd{t}\dd{x}\\
			&\lesssim \int \frac{\abs{t'}^{1/2}}{\abs{x}^{d+1/3}}\dd{x}
			\lesssim \frac{\abs{t'}^{1/2}}{\norm{z'}}\lesssim 1.
		\end{align}
		In the domain of $2\abs{t'}\ge \abs{t}$, 
		we estimate the integrals of 
		$\abs{\gamma_3((0,0,t')\circ t)}$ and 
		$\abs{\gamma_3(z)}$ respectively.
		Thanks to the fact $\abs{t+t'}\le 3\abs{t'}$, 
		\begin{align}
			\int_{2\abs{t'}\ge \abs{t},
				\abs{x}^{1/3}\ge \norm{z}/2\gtrsim\norm{z'}}
			\abs{\gamma_3((0,0,t')\circ t)}\dd{z}
			&\lesssim \int\frac{1}{\norm{z}^{3d+1}(\abs{y}+\abs{t+t'}^{1/2})^{d+1}}
			\dd{z}\\
			&\lesssim 
			\int \frac{1}{\abs{x}^{d+1/3}\abs{t+t'}^{1/2}}\dd{t}\dd{x}\\
			&\lesssim \int \frac{\abs{t'}^{1/2}}{\abs{x}^{d+1/3}}\dd{x}
			\lesssim 1
		\end{align}
		holds.

		An analogous computation shows the inequality 
		\eqref{eq:Hormander condition for ball} holds 
		for the adjoints of $\Gamma_2$ and $\Gamma_3$.
	\end{proof}

	\begin{proposition}\label{prop: Hormander condition 4}
		The inequality \eqref{eq:Hormander condition for ball}
		holds for $K=\Gamma_4$ and its adjoint as $c=c_1$.
	\end{proposition}
	\begin{proof}
		We estimate 
		the integral of $\abs{\gamma_4(z'\circ z)-\gamma_4(z)}$
		over the domain of $\norm{z'}\le c_1\norm{z}$ 
		with Proposition \ref{prop:pointwise estimate gamma4}.
		We focus on the domain of $\abs{x}^{1/3}+\abs{t}^{1/2}\ge \norm{z}/2$.
		Fix $\nu\in (0,1)$ small enough.
		We have 
		\begin{align}
			\abs*{\gamma_4((x',0,t')\circ z)-\gamma_4((0,0,t')\circ z)}
			&\lesssim \frac{\abs{x'}}{\norm{z}^{3d+6-\nu}\abs{y}^{d-1+\nu}},\\
			\abs*{\gamma_4((0,0,t')\circ z)-\gamma_4(z)}
			&\lesssim \frac{\abs{t'}}{\norm{z}^{3d+5-\nu}\abs{y}^{d-1+\nu}}.
		\end{align}
		Integrate them with respect to $y$ first. Then we see 
		\begin{align}
			\int_{\abs{x}^{1/3}+\abs{t}^{1/2}\ge \norm{z}/2\gtrsim \norm{z'}}
			\abs*{\gamma_4((x',0,t')\circ z)-\gamma_4((0,0,t')\circ z)}\dd{z}
			&\lesssim 1,\\
			\int_{\abs{x}^{1/3}+\abs{t}^{1/2}\ge \norm{z}/2\ge \norm{z'}/2c_1}
			\abs*{\gamma_4((0,0,t')\circ z)-\gamma_4((0,0,t')\circ z)}\dd{z}
			&\lesssim 1.
		\end{align}

		Let us estimate $\abs*{\gamma_4(z'\circ z)-\gamma_4((x',0,t')\circ z)}$.
		The idea of computation is shared with 
		Proposition \ref{prop: Hormander condition 2, 3}:
		we decompose the domain into 
		the part of $\abs{y}\ge 2\abs{y'}$ and the other.
		Suppose that $\abs{y}\ge 2\abs{y'}$.
		Then 
		$\abs{y+\theta y'}\ge \abs{y}-\abs{y'}\ge \abs{y}/2$
		for $\theta\in[0,1]$.
		Hence,
		\begin{align}
			\abs*{\gamma_4(z'\circ z)-\gamma_4((x',0,t')\circ z)}
			&\lesssim 
			\int_0^1 
			\frac{\abs{y'}}{\norm{z}^{3d+3-\nu}\abs{y+\theta y'}^{d+\nu}}\dd{\theta}\\
			&\lesssim 
			\frac{\abs{y'}}{\norm{z}^{3d+3-\nu}\abs{y}^{d+\nu}},
		\end{align}
		which is integrable with respect to $y$,
		and therefore
		\begin{align}\MoveEqLeft
			\int_{\abs{y}\ge 2\abs{y'},\abs{x}^{1/3}+\abs{t}^{1/2}\ge \norm{z}\gtrsim \norm{z'}}
			\abs*{\gamma_4(z'\circ z)-\gamma_4((x',0,t')\circ z)}\dd{z}\\
			&\le \int \frac{\abs{y'}}{\norm{(x,0,t)}^{3d+3-\nu}\abs{y}^{d+\nu}}\dd{z}\\
			&\lesssim \int \frac{\abs{y'}^{1-\nu}} 
				{(\abs{x}^{1/3}+\abs{t}^{1/2})^{3d+3-\nu}}\dd{x}\dd{t}\\
			&\lesssim \frac{\abs{y'}^{1-\nu}}{\norm{z'}^{1-\nu}}
			\lesssim 1.
		\end{align}
		Suppose that $2\abs{y'}\ge \abs{y}$,
		By the triangle inequality,
		\begin{align}
			\abs*{\gamma_4(z'\circ z)-\gamma_4((x',0,t')\circ z)}
			&\lesssim 
				\frac{1}{\norm{(x,0,t)}^{3d+3-\nu}\abs{y+y'}^{d-1+\nu}}
					+\frac{1}{\norm{(x,0,t)}^{3d+3-\nu}\abs{y}^{d-1+\nu}}.
		\end{align}
		From $\abs{y+y'}\le \abs{y}+\abs{y'}\le 3\abs{y'}$,
		we have 
		\begin{align}\MoveEqLeft
			\int_{2\abs{y'}\ge \abs{y},
			\abs{x}^{1/3}+\abs{t}^{1/2}\ge \norm{z}\gtrsim \norm{z'}}
			\abs*{\gamma_4(z'\circ z)-\gamma_4((x',0,t')\circ z)}
			\dd{z}\\
			&\lesssim 
				\int_{\abs{x}^{1/3}+\abs{t}^{1/2}\gtrsim \norm{z'}}
				\frac{\abs{y'}^{1-\nu}}{(\abs{x}^{1/3}+\abs{t}^{1/2})^{3d+3-\nu}}
				\dd{x}\dd{t}\\
			&\lesssim 
				\frac{\abs{y'}^{1-\nu}}{\norm{z'}^{1-\nu}}
			\lesssim 1.
		\end{align}
		An analogous computation shows the inequality 
		\eqref{eq:Hormander condition for ball} holds 
		for the adjoint of $\Gamma_4$.
	\end{proof}
\subsection{Anisotropic estimates}
	\label{sec:proof of LyLxy estimate}

	To prove anisotropic estimates, we need the theory of
	singular integrals for vector-valued functions.
	The following theorem is Theorem 5.17 in \cite{duo2001fourier},
	whose proof is sketched there.
	Let $\bddlin(A,B)$ denote the set of bounded linear operators 
	from $A$ into $B$ for Banach spaces $A$ and $B$.
	\begin{theorem}
		\label{thm:SI vector-valued}
		Let $A, B$ be separable Banach spaces and $p_0\in (1,\infty)$.
		Assume that a bounded linear operator
		$T\colon L^{p_0}(\bbR^d; A)\to L^{p_0}(\bbR^d; B)$
		has an integral kernel $K$ 
		that is defined on $\bbR^d\times\bbR^d\setminus\Delta$
		and takes values in $\bddlin(L^{p_0}(\bbR^d; A),L^{p_0}(\bbR^d; B))$ 
		such that if $f\in L^{p_0}(\bbR^d; A)$ has compact support, then
		\begin{equation}
			Tf(x)=\int_{\bbR^d}K(x,y)f(y)\dd{\mu(y)},
			\qquad x\notin\supp{f}
		\end{equation}
		in the mean of Bochner integral.
		If there is a constant $c>1$ such that
		\begin{equation}
			\int_{\abs{x-y}\ge c \abs{y'-y}}
			\norm{K(x,y)-K(x,y')}_{\bddlin(A,B)}\dd{\mu(x)}
		\end{equation}
		is bounded by some constant independent of $y$ and $y'$,
		then $T$ is weak $(1,1)$ and strong $(p,p)$ when $1<p<p_0$.

		Furthermore, If in addition,
		\begin{equation}
			\int_{\abs{x-y}\ge c \abs{y'-y}}
			\norm{K(y,x)-K(y',x)}_{\bddlin(A,B)}\dd{\mu(x)},
		\end{equation}
		is bounded by some constant independent of $y$ and $y'$,
		then $T$ is strong $(p,p)$ for all $p\in (1,\infty)$.
	\end{theorem}

	Let $1<p<\infty$ and $i=1,2$.
	Setting
	\begin{align}
		(\tilde{T_i}(y,y')f)(x,t)
		=\int_{\bbR^{d+1}}K_i(x,y,t,x',y',t')f(x',t')\dd{x'}\dd{t'},
	\end{align}
	we can write
	\begin{equation}
		T_if(z)=\int_{\bbR^d}
		(\tilde{T_i}(y,y')f(\dummydot,y',\dummydot))(x,t)\dd{y'}.
	\end{equation}
	Since its $L^{p}$-boundedness is already proved,
	it is enough to check the standard estimate, i.e.,
	\begin{equation}
		\int_{\abs{y-y'}\ge c_1\abs{y'-y''}}
		\norm{\tilde{T_i}(y,y')-\tilde{T_i}(y,y'')}_{L^p_{x,t}\to L^p_{x,t}}
		\dd{y}<C
	\end{equation}
	for some constant $C$ independent of $y'$ and $y''$.
	When $\abs{y-y'}\ge c_1\abs{y'-y''}$,
	the inequality 
	$\norm{(z')\inv \circ z}\ge c_1\norm{(x',y'',t')\inv \circ z'}$
	holds and thus
	\begin{equation}
		\abs{\Gamma_1(z,z')-\Gamma_1(z,(x',y'',t'))}
		\lesssim \frac{\abs{y-y''}}{\norm{(z')\inv\circ z}^{4d+3}}.
	\end{equation}
	The Shur test is useful now.
	Integrating
	\begin{equation}
		\norm*{\frac{\abs{y-y''}}{\norm{(z')\inv\circ z}^{4d+3}}}_{L_x^1}
		\lesssim \frac{\abs{y-y'}}{
			(\abs{y'-y''}+\abs{t-t'}^{1/2})^{3(d/3+1)}
		}
	\end{equation}
	with respect to $t$, we get
	\begin{equation}\label{eq:difference of Gamma1 anisotropic}
		\norm*{\Gamma_1(z,z')-\Gamma_1(z,(x',y'',t'))}_{L_{x,t}^1}
		\lesssim \frac{\abs{y'-y''}}{\abs{y-y'}^{d+1}}.
	\end{equation}
	The estimate of adjoint version is shown by similar calculation,
	i.e.,
	\begin{equation}\label{eq:difference of adjoint Gamma1 anisotropic}
		\norm*{\Gamma_1(z,z')-\Gamma_1(z,(x',y'',t'))}_{L_{x',t'}^1}
		\lesssim \frac{\abs{y'-y''}}{\abs{y-y'}^{d+1}}.
	\end{equation}
	The estimates \eqref{eq:difference of Gamma1 anisotropic} and 
	\eqref{eq:difference of adjoint Gamma1 anisotropic} yield that
	\begin{equation}
		\norm{\tilde{T_1}(y,y')-\tilde{T_1}(y,y'')}_{L^p_{x,t}\to L^p_{x,t}}
		\lesssim \frac{\abs{y'-y''}}{\abs{y-y'}^{d+1}}
	\end{equation}
	for $p\in \sqb{1,\infty}$
	where $\abs{y-y'}\ge c_1\abs{y'-y''}$.
	The adjoint of $\tilde{T_1}$ defined as
	$\tilde{T_1}^*(y,y')=\tilde{T_1}(y',y)$ also enjoys the same estimate.

	Let us prove the same estimate for $T_3$, which implies 
	the estimate for the better operator $T_2$.
	\begin{align}\MoveEqLeft
		\abs{\Gamma_3(z,z')-\Gamma_3(z,(x',y'',t'))}\\
		& \le \abs{y'-y''}
		\int_0^1
		\abs{t-t'}\abs{\nabla_x\gamma_3(
		(x',y'+\theta(y''-y'),t')\inv \circ z)}
			+\abs{\nabla_y\gamma_3((x',y'+\theta(y''-y'),t')\inv \circ z)}
			\dd{\theta}\\
		& \lesssim
		\int_0^1
		\frac{\abs{y'-y''}}{
		\norm{(x',y'+\theta(y''-y'),t')\inv \circ z}^{3d+1}}
		\frac{1}{(\abs{y-(y'+\theta(y''-y'))}+\abs{t-t'}^{1/2})^{d+2}
		}\dd{\theta}\\
		& \lesssim
		\frac{\abs{y'-y''}}{(\abs{y-y'}+\abs{t-t'}^{1/2})^{d+2}}
		\int_0^1\frac{1}{
		\norm{(x-x'-(t-t')(y'+\theta(y''-y')),y-y',t-t')}^{3d+1}
		}\dd{\theta}
	\end{align}
	and thus by Fubini's theorem,
	\begin{align}\MoveEqLeft
		\int_{\bbR^{d+1}}\abs{\Gamma_3(z,z')-\Gamma_3(z,(x',y'',t'))}
		\dd{x}\dd{t} \\
		&\lesssim 
			\int_{\bbR}
			\frac{\abs{y'-y''}}{(\abs{y-y'}+\abs{t-t'}^{1/2})^{d+2}}
			\int_0^1\int_{\bbR^d}\frac{1}{
			\norm{(x,y-y',t-t')}^{3d+1}
			}\dd{x}\dd{\theta}\dd{t}
		\\
		&\qquad(x=(\abs{y-y'}+\abs{t-t'}^{1/2})^{3}\xi)\\
		& \lesssim \int_{\bbR^d}\frac{1}{(\abs{\xi}+1)^{d+1/3}}\dd{\xi}
		\int_{\bbR}
			\frac{\abs{y'-y''}}{(\abs{y-y'}+\abs{t-t'}^{1/2})^{d+3}}\dd{t} \\
		&\lesssim\frac{\abs{y'-y''}}{\abs{y-y'}^{d+1}}
		\int_{\bbR^d}\frac{1}{(\abs{\tau}^{1/2}+1)^{d+3}}\dd{\tau}
		\qquad (t-t'=\abs{y-y'}^2\tau).
	\end{align}
	The estimates for $T_2^*$ and $T_3^*$ 
	we need are shown in the similar way.

	Finally we consider the estimate for $T_4$.
	We have 
	\begin{align}\MoveEqLeft
		\abs{\Gamma_4(z,z')-\Gamma_4(z,(x',y'',t'))} \\
		 & \lesssim
		\abs{y'-y''}\int_0^1
		\abs{t-t'}\abs{\nabla_x\gamma_4((x',y'+\theta(y''-y'),t')\inv \circ z)}
		\\
		&\qquad 
		+\abs{\nabla_y\gamma_4((x',y'+\theta(y''-y'),t')\inv \circ z)}
		\dd{\theta}\\
		&\lesssim\abs{y'-y''}\int_0^1
		 \frac{1}{\norm{(x',y'+\theta(y''-y'),t')\inv \circ z}^{3d+3-\nu}}
		\cdot \frac{1}{\abs{y-(y'+\theta(y''-y'))}^{d+\nu}}
		\dd{\theta}
	\end{align}
	for any $\nu>0$. Let $\nu$ be small enough.
	The difference
	$\abs{\Gamma_4(z,z')-\Gamma_4(z,(x',y'',t'))}$
	is integrable with respect to $(x, t)$, and 
	\begin{equation}
		\norm{\Gamma_4(z,z')-\Gamma_4(z,(x',y'',t'))}_{L^1_{x,t}}
		\lesssim \int_0^1
			\frac{\abs{y'-y''}}{\abs{y-(y'+\theta(y''-y'))}^{d+1}}\dd{\theta}
		\lesssim \frac{\abs{y'-y''}}{\abs{y-y'}^{d+1}}
	\end{equation}
	holds if $\abs{y-y'}\ge 2\abs{y''-y'}$.
	Similar computation yields the estimate for $T_4^*$,
	and consequently we have Theorem \ref{thm:main result L}.

\section{Estimates for solutions to stationary problem}
	\label{sec:proof for stationary solutions}

	In this section, we shall prove the main result of this paper
	for the stationary problem.
	Take a function $f\in C_{\mathrm{c}}^{\infty}(\bbR^{2d})$ and
	Let $\chi\colon\bbR\to \bbR$ be a smooth cut-off function
	with compact support
	such that $\chi(t)=1$ if $\abs{t}\le 2$.
	Define $\chi_R$ as $\chi_R(t)=\chi(t/R)$ for $R>0$.
	Let $u_R$ denote the solution of $Lu_R=f(x,y)\chi_R(t)$
	defined by
	\begin{equation}
		u_R(x,y,t)
		=\int_{\bbR^{2d+1}}
		f(x',y')\chi_R(t')\gamma((z')\inv\circ z)\dd{z'}
	\end{equation}
	and write
	\begin{equation}
		u_{\infty}(x,y)
		=\int_{\bbR^{2d+1}}f(x',y')\gamma((z')\inv\circ z)\dd{z'}.
	\end{equation}
	Then $u_{\infty}$ is approximated by the $R$-average
	\begin{equation}
		U_R(x,y)=\frac{1}{2R}\int_{-R}^{R}
		u_R(x,y,t)\dd{t}
	\end{equation}
	in $L^\infty(\bbR^{2d})$. Indeed, 
	\begin{align}\MoveEqLeft
		\sup_{\abs{t}\le R}
		\abs{u_R(x,y,t)-u_{\infty}(x,y)}\\
		& \lesssim \sup_{\abs{t}\le R}
		\int \frac{1}{\norm{(x,y,t)\circ (x',y',t')\inv}^{4d}}
		\abs{f(x',y')}\prn{1-\chi_R(t')}\dd{z'}\\
		& \lesssim \norm{f}_{L^1}\sup_{\abs{t}\le R}
		\int_{\abs{t-t'}\ge R}\frac{1}{\abs{t-t'}^{2d}}\dd{t'}
		\lesssim\frac{\norm{f}_{L^1}}{R^{2d-1}}
		\label{al:u_infty estimate}
	\end{align}
	and hence
	\begin{equation}
		\abs{U_R(x,y)-u_{\infty}(x,y)}
		\le \frac{1}{2R}\int_{-R}^R \abs{u_R(x,y,t)-u_{\infty}(x,y)}\dd{t}
		\lesssim \frac{\norm{f}_{L^1}}{R^{2d-1}}.
	\end{equation}
	Since $u_R$ belongs to $L^{\infty}(\bbR^{2d+1})$ as 
	we proved in Proposition \ref{prop:Lp estimate for lower order term},
	the estimate \eqref{al:u_infty estimate} 
	also gives $u_{\infty}\in L^{\infty}(\bbR^{2d})$.

	Next, we prove the estimate 
	$\norm{\laplacian_{y}u_{\infty}}_{L_y^{q}L_{x}^{p}}
		\lesssim \norm{f}_{L_y^{q}L_{x}^{p}}$.
	Observe that 
	\begin{equation}
		\laplacian_yU_R=\frac{1}{2R}\int_{-R}^R \laplacian_yu_R \dd{t}.
	\end{equation}
	Indeed, for any test function $\phi\in C\sub{c}^{\infty}(\bbR^{2d})$, 
	\begin{align}
		\agb{\laplacian_yU_R,\phi}
		&=\agb{U_R,\laplacian_y\phi}
		=\frac{1}{2R}\int_{-R}^R \agb{u_R,\laplacian_y\phi} \dd{t}\\
		&=\frac{1}{2R}\int_{-R}^R \agb{\laplacian_y u_R,\phi}\dd{t}
		=\agb*{\frac{1}{2R}\int_{-R}^R\laplacian_y u_R\dd{t},\phi}
	\end{align}
	due to the fact $u_R,\laplacian_yu_R\in L^2(\bbR^{2d+1})$ and 
	Fubini's theorem.
	Again, take a test function $\phi\in C\sub{c}^{\infty}(\bbR^{2d})$.
	We have
	\begin{equation}
		\agb{\laplacian_yu_{\infty},\phi}
		=\agb{u_{\infty},\laplacian_y\phi}
		=\lim_{R\to\infty}\agb{U_R,\laplacian_y\phi}
		=\lim_{R\to\infty}\agb{\laplacian_yU_R,\phi}
	\end{equation}
	but H\"{o}lder's inequality yields
	\begin{equation}
		\norm{\laplacian_yU_R}_{L_y^{q}L_{x}^{p}}
		\lesssim \frac{1}{R^{1/p}} \norm{\laplacian_yu_R}_{L_y^{q}L_{x,t}^{p}}
		\lesssim \frac{1}{R^{1/p}} \norm{f\chi_R}_{L_y^{q}L_{x,t}^{p}}
		\lesssim \norm{f}_{L_y^{q}L_{x}^{p}}
	\end{equation}
	and therefore
	$\norm{\laplacian_{y}u_{\infty}}_{L_y^{q}L_{x}^{p}}
		\lesssim \norm{f}_{L_y^{q}L_{x}^{p}}$.
	Similarly we obtain the estimate
	\begin{equation}
		\norm{\abs{\nabla_x}^{2/3}u_{\infty}}_{L_y^{q}L_{x}^{p}}
		+\norm{\abs{\nabla_x}^{1/3}\nabla_yu_{\infty}}_{L_y^{q}L_{x}^{p}}
		\lesssim \norm{f}_{L_y^{q}L_{x}^{p}}.
	\end{equation}

	Finally, we claim that $u_{\infty}$ is a weak solution of $-\calA{u}=f$.
	The approximating function $U_R$ satisfies
	\begin{align}
		\laplacian_yU_R-y\cdot \nabla_x{U_R}
		& =\frac{1}{2R}\int_{-R}^R \prn{\partial_tu_R-f\chi_R}\dd{t} \\
		& =\frac{u_R(\dummydot,\dummydot,R)-u_R(\dummydot,\dummydot,-R)}{2R}-f.
	\end{align}
	Since $\abs{u_R(\dummydot,\dummydot,\pm R)}$ is	uniformly bounded,
	taking the limit yields that $-\calA{u_{\infty}}=f$ holds in the weak sense.
	
	\bigskip
	\noindent
	\textbf{Acknowledgement}:
	The author expresses my deepest gratitude to my academic advisor,
	Professor \relax{Yasunori Maekawa}, whose expert guidance and patience were
	critical in the development of this work.
	His unwavering dedication to academic excellence
	inspired me throughout.

	\bigskip
	This research did not receive any specific grant from funding agencies 
	in the public, commercial, or not-for-profit sectors.

\bibliographystyle{abbrv}
\bibliography{bib2025a}
\end{document}